\documentclass[reqno,oneside,12pt]{amsart}

\usepackage[T1]{fontenc}
\usepackage{times,dsfont}
\usepackage{amssymb,epsfig,verbatim}

\usepackage{mathtools}
\usepackage{color}
\theoremstyle{plain}

\newtheorem{thm}{Theorem}[section]

\newtheorem{pro}[thm]{Proposition}
\newtheorem{lem}[thm]{Lemma}
\newtheorem{proposition-principale}[thm]{Proposition principale}
\newtheorem{thm-principal}{Main Theorem}
\newtheorem{defi}[thm]{Definition}

\theoremstyle{definition}
\newtheorem{que}[thm]{Question}
\newtheorem{eg}[thm]{Example}
\newtheorem{rem}[thm]{Remark}

\newenvironment{defi-G}
{\noindent{\bf Definition.}\it }{}

\newenvironment{thm-M}
{\noindent{\bf Main Theorem.}\it }{}

\newenvironment{thm-C}
{\noindent{\bf Theorem C.}\it }{\\ }

\newenvironment{thm-A}
{\noindent{\bf Theorem A.}\it}{\\ }

\newenvironment{thm-B}
{\noindent{\bf Theorem B.}\it}{\\ }

\newenvironment{thm-BB}
{\noindent{\bf Theorem B'.}\it}

\def\C{\mathbf{C}}
\def\R{\mathbf{R}}
\def\Q{\mathbf{Q}}

\def\Z{\mathbf{Z}}
\def\N{\mathbf{N}}
\def\bfk{\mathbf{k}}
\def\bfK{\mathbf{K}}

\def\Qbar{{\overline{\mathbf{Q}}}}

\def\bfx{\mathbf{x}}

\def\Norm{{\mathrm{Norm}}}

\def\HH{{\mathbb{H}}}
\def\bfe{{\mathbf{e}}}
\def\bfh{{\mathbf{h}}}
\def\bfc{{\mathbf{c}}}
\def\bfv{{\mathbf{v}}}

\def\bp{{\mathrm{bp}}}

\def\Ima{{\mathrm{Im}}}
\def\Num{{\mathrm{Num}}}

\def\P{\mathbb{P}}
\def\A{\mathbb{A}}
\def\F{{\mathbb{F}}}

\def\sfO{{\sf{O}}}

\def\ZZ{\mathcal{Z}}
\def\ZC{\mathsf{Z}}

\def\stable{{\mathrm{sl}}}

\def\Aut{{\sf{Aut}}}

\def\Bir{{\sf{Bir}}}

\def\Isom{{\sf{Iso}}}

\def\PGL{{\sf{PGL}}\,}
\def\PSL{{\sf{PSL}}\,}
\def\GL{{\sf{GL}}\,}

\def\Iso{{\sf{Iso}}\,}

\def\Gm{{\mathbb{G}}_{\sf{m}}}
 
\def\SL{{\sf{SL}}\,}

\newcommand{\Id}{{\rm Id}}

\def\Pic{{\mathrm{Pic}}}
\def\Cr{{\mathrm{Cr}}}

\def\dist{{\sf{dist}}}

\addtocounter{section}{0}             
\numberwithin{equation}{section}       

\begin{document}

\setlength{\baselineskip}{0.56cm}        
%
%

\title{Distortion in Cremona groups}
\date{February 29, 2020}
\author{Serge Cantat and Yves de Cornulier}
\address{CNRS and Univ Rennes, IRMAR - UMR 6625, F-35000 Rennes}
\email{serge.cantat@univ-rennes1.fr}

\address{
CNRS and Univ Lyon, Univ Claude Bernard Lyon 1, Institut Camille Jordan, 43 blvd. du 11 novembre 1918, F-69622 Villeurbanne}
\email{cornulier@math.univ-lyon1.fr}
\thanks{S.C.\ was supported by the French Academy of Sciences (Fondation del Duca). Y.C.\ was supported by ANR Gamme (ANR-14-CE25-0004). }


%
%

%
%

%
%

\begin{abstract} 
We study the distortion of elements in two-dimensional Cremona groups over algebraically closed fields of characteristic zero. We obtain the following trichotomy: non-elliptic elements (i.e., those whose powers have unbounded degree) are undistorted, and elliptic elements have a double exponential distortion when they are virtually unipotent or an exponential distortion otherwise.
\end{abstract}

\maketitle

\setcounter{tocdepth}{1}

\section{Introduction}

Let $\bfk$ be an algebraically closed field. 
The goal of this paper is to study the distortion in the Cremona group $\Bir(\P^2_\bfk)$. 
We characterize distorted elements, and study their distortion function. The three 
main tools are: (1) an upper bound on the distortion which is obtained via height estimates, using 
basic number theory (this holds in arbitrary dimension); (2) a result of Blanc and D\'eserti concerning 
base points of birational transformations of the plane;  (3) a non-distortion result for parabolic elements  in $\Bir(\P^2_\bfk)$, 
obtained via Noether inequalities and the study of the action of $\Bir(\P^2_\bfk)$ on the Picard-Manin space (an infinite-dimensional
hyperbolic space). This third step sheds new light on the geometry of the action of $\Bir(\P^2_\bfk)$ on this hyperbolic space. 

\subsection{Distortion}

If $f$ and $g$ are two real-valued functions on $\R_+$, we write $f\preceq g$ if there exist three positive constants $C$, $C'$, $C''$ 
such that $f(x)\le Cg(C'x)+C''$ for all $x\in \R_+$. We write $f\simeq g$ when $f\preceq g\preceq f$.

\begin{defi}Let $G$ be a group. If $S$ and $T$ are two subsets of $G$ containing the neutral element~$1$, we 
write $S\preceq T$ if $S\subset T^k$ for some integer $k\geq 0$, and $S\simeq T$ if $S\preceq T\preceq S$.
Let $c$ be an element of $G$. Let $S$ be a finite symmetric subset of $G$ containing $1$; if  the 
subgroup $G_S$ generated by $S$ contains $c$, we define the {\bf{distortion function}}
\[
\delta_{c,S}(n)=\sup\{m\in\N:c^m\in S^n\}
\]
\end{defi}
By definition, $\delta_{c,S}(n)=\infty$ if and only if $c$ has finite order.
Clearly, if $S\subset T$ then $\delta_{c,S}\le \delta_{c,T}$. Also, $\delta_{c,S^k}(n)=\delta_{c,S}(kn)$.
In particular, if $S\subset T^k$, then $\delta_{c,S}\le\delta_{c,T}(kn)$. If   $S\preceq T$,   it follows that  $\delta_{c,S}\preceq\delta_{c,T}$, and
if $S\simeq T$ then $\delta_{c,S}\simeq\delta_{c,T}$.

If $S$ and $T$ both generate $G$ then $S\simeq T$ and $\delta_{c,S}\simeq\delta_{c,T}$. Thus, when $G$ is finitely generated,
 the $\simeq$-equivalence class of the distortion function only depends on $(G,c)$, not on the finite generating subset; 
it is called the distortion function of $c$ in $G$, and is denoted $\delta_c^G$, or simply $\delta_c$. 
The element $c$ is called {\bf{undistorted}} if $\delta_c(n)\preceq n$, and {\bf{distorted}} otherwise. 

\begin{eg}
Fix a pair of integers $k, \ell \geq 2$. 
In the Baumslag-Solitar group $B_k=\langle t,x\vert \quad t xt^{-1}=x^k\rangle$, we have $\delta^{B_k}_x(n)\simeq\exp(n)$.
In the "double" Baumslag-Solitar group $B_{k,\ell}$ (see~\cite{Gromov:AIIG} and \S~\ref{par:2Examples}), one finds double exponential distortion.
\end{eg}

It is natural to consider distortion in groups that are not finitely generated. We say that an element $c\in G$ 
is {\bf{undistorted}} if $\delta_c^H(n)\simeq n$ for every finitely generated subgroup $H$ of $G$ containing $c$.
Changing $H$ may change the distortion function $\delta^H_c$; for instance, if $c$ is not a torsion element, 
it is undistorted in $H=c^\Z$ but may be distorted in larger groups. Also, there are examples of pairs $(G,c)$ 
such that $c$ becomes more and more distorted, in larger and larger subgroups of $G$ (see \S~\ref{par:2Examples}). Thus, we have 
a good notion of distortion, but the distortion is not measured by an equivalence class of a function " $\delta_c^G$ ". 

We shall say that the {\bf{distortion type}} (or class) of $c$ in $G$ is at least $f$ if there is a finitely generated 
subgroup $H$ containing $c$ with $f\preceq \delta^H_c$, and is at most $g$ if $\delta^H_c\preceq g$ for 
all finitely generated subgroup $H$ containing $c$. If the distortion type is at least $f$ and at most $f$
simultaneously, we shall say that $f$ is the distortion type of $c$. For instance, $c$ may be exponentially, 
or doubly exponentially distorted in $G$. 
 
 \begin{eg} 
Let $\bfk$ be a field. Let $c$ be an element of  the general linear group $\GL_d(\bfK)$;  we have one of the following (see \cite{LMR2, LMR1} and \S~\ref{par:h-and-d})
\begin{itemize}
\item $c$ is not virtually unipotent, i.e. at least one of its eigenvalues in an algebraic closure of $\bfk$ is not a root of unity, and then $c$ is undistorted;
\item $c$ is virtually unipotent of infinite order, and then $\delta_c(n)\simeq\exp(n)$ (this occurs only if $\bfk$ has characteristic zero);
\item $c$ has finite order.
\end{itemize}
The dimension $d\ge 0$ does not intervene in this description. In contrast, the unipotent elementary matrix 
$e_{12}(1)=\Id + \delta_{1,2}$ is undistorted in $\SL_2(\Z)$ but has exponential distortion in $\SL_d(\Z)$ for $d\ge 3$. 
\end{eg}

\subsection{Distortion in Cremona groups}

Distortion in  groups of homeomorphisms is an active subject (see \cite{Avila:distortion, Calegari-Freedman, LeRoux-Mann, Militon:2013, Militon:2014}). For
instance, in the group of homeomorphisms of the sphere ${\mathbb{S}}^d$, every element is distorted. 
Our goal in this paper is to study distortion in groups of birational 
transformations. 

If $M$ is a projective variety over a field $\bfk$, we denote by 
$\Bir(M_\bfk)$ its group of birational transformations over $\bfk$. When $M$
is the projective space $\P^m_\bfk$, this group is the {\bf{Cremona group}} in 
$m$ variables
$\Cr_m(\bfk)=\Bir(\P^m_\bfk)=\Bir(\A^m_\bfk)$.
The problem is to describe the elements of $\Bir(M_\bfk)$ which are distorted in $\Bir(M_\bfk)$, 
and to estimate their distortion functions. 
\subsubsection{Degree sequences} \label{par:degrees}
Let $H$ be a hyperplane section of $M$, for some fixed embedding $M\subset\P^N_\bfk$. The {\bf{degree}} of a birational transformation $f\colon M\dasharrow M$
with respect to the polarization $H$ is the intersection product $\deg_H(f)=H^{m-1}\cdot f^{*}(H)$,
where $m=\dim(M)$. 
When $M$ is $\P^m_\bfk$ and $H$ is a hyperplane, then $\deg_H(f)$ is the degree of the
homogeneous polynomial functions $f_i$, without common factor of positive degree,  such that $f=[f_0:\cdots :f_m]$ in homogeneous coordinates. 

The degree function is almost submultiplicative (see \cite{Dinh-Sibony:2005Annals, NguyenBD:2017, TTTruong}): there is a constant
$C_{M,H}$ such that for all $f$ and $g$ in $\Bir(M_\bfk)$
\begin{equation}\label{eq:submultiplicative-degrees}
\deg_H(f\circ g)\leq C_{M,H} \deg_H(f)\deg_H(g).
\end{equation}
Thus, we can define the   {\bf{dynamical degree}} $\lambda_1(f)$ by
$
\lambda_1(f)=\lim_{n\to +\infty} (\deg_H(f^n)^{1/n}).
$
By definition, $\lambda_1(f)\geq 1$, and the following well-known lemma implies that $\lambda_1(f)=1$ when 
$f$ is distorted (see Section~\ref{par:dyna-degree}).

\begin{lem}\label{lem:stable-length}
Let $G$ be a group with a finite symmetric generating subset $S$. Let $\vert w\vert$ denote the word length of $w\in G$ with respect 
to the generating subset $S$. Then,
\begin{enumerate}
\item $\vert\cdot \vert$ is sub-additive: $\vert vw\vert \leq \vert v \vert+\vert w\vert$;
\item the {\bf{stable length}} $\stable(c):=\lim_{n\to \infty} \frac{1}{n}\vert c^n\vert$ is a well-defined element of $\R_+$;
\item $c$ is distorted if and only if $\stable(c)=0$. 
\end{enumerate}
\end{lem}

\subsubsection{Distortion in dimension $2$} 
Assume, for simplicity, that the field $\bfk$ is algebraically closed. 
Typical elements of $\Cr_d(\bfk)$ have  dynamical degree $>1$. 
At the opposite, we have the notion of {\bf{algebraic elements}}. A birational transformation 
$f\colon M\dasharrow M$ is {\bf{algebraic}}, or {\bf{bounded}}, if  $(\deg_H(f^{n}))_{n\ge 0}$ is
a bounded sequence of integers; by a theorem of Weil (see \cite{Weil}), $f$ is bounded if and
only if there exists a projective variety $M'$, a birational map $\varphi\colon M'\dasharrow M$, and an integer $m>0$, 
such that $\varphi^{-1}\circ f^m \circ \varphi$ is an element of $\Aut(M')^0$ (the connected
component of the identity in the group of automorphisms $\Aut(M')$). In the case of surfaces, bounded
elements are also called {\bf{elliptic}}; we shall explain this terminology in Section~\ref{prundi}. 

\begin{thm}\label{main}
Let $\bfk$ be a field. If an element $f\in \Cr_2(\bfk)$ is distorted, then $f$
is elliptic. If $\bfk$ is algebraically closed and of characteristic $0$, and $f\in \Cr_2(\bfk)$
 is elliptic and of infinite order, then:
\begin{itemize}
\item if some positive power of $f$ is conjugate to a unipotent automorphism 
 of $\P^2_\bfk$, then $f$ has double exponential distortion;
\item otherwise, $f$ has exponential distortion.
\end{itemize}
\end{thm} 

The first assertion extends to $\Bir(X)$ for all projective surfaces (see Theorems~\ref{thm:Halphen-Undistorted} and~\ref{thm:Jonquieres-Undistorted}), but the second does not.
 For instance, if $X$ is 
a complex abelian surface and $\Aut(X)$ has only finitely many connected components, 
every translation of infinite order is undistorted and elliptic.
 
Consider, in $\Cr_2(\bfk)$, the element $(x,y)\stackrel{s}\mapsto (x,xy)$; it is not elliptic and by the above theorem, it is not distorted in $\Cr_2(\bfk)$. On the other hand, the natural embedding $\Cr_2(\bfk)\subset \Cr_3(\bfk)$ maps it to
 $(x,y,z)\mapsto (x,xy,z)$, which is exponentially distorted in $\Bir(\A^3_\bfk)$, while its degree growth remains linear. Thus Theorem \ref{main} is specific to the projective plane. 

\begin{que} (see Section~\ref{par:h-and-d})

(A)  In Theorem \ref{main}, can we remove the restriction concerning the characteristic or the algebraic closedness of the field $\bfk$?
 
(B) Can we find an element of infinite order with more than double exponential distortion in the Cremona group $\Cr_m(\C)$, for some $m\geq 3$?  
\end{que}

\subsection{Hyperbolic spaces, horoballs, and distortion} 

Our proof of Theorem~\ref{main} makes use of the action of $\Cr_2(\bfk)$ on an infinite dimensional hyperbolic space $\HH_\infty$,
already at the heart of several articles (see~\cite{Cantat:SLC}). There are elements $f$ of $\Cr_2(\bfk)$ acting as parabolic isometries on $\HH_\infty$, 
with a unique fixed point $\xi_f$ at the boundary of the hyperbolic space. We shall show that the orbit of a sufficiently small
horoball centered at $\xi_f$ under the action of $\Cr_2(\bfk)$ is made of a family of pairwise disjoint horoballs. We refer
to Theorem~C in Section~\ref{par:Par-Horoball} for that result.  Theorem~B, proved in Section~\ref{prundi}, 
is a general result for groups acting by isometries
on hyperbolic spaces that provides a control of the distortion of parabolic elements. 

\subsection{Remark and Acknowledgement}

One step towards Theorem~\ref{main} is to prove that the so-called Halphen twists of $\Cr_2(\bfk)$ (a certain type of parabolic elements) are
not distorted. Blanc and Furter obtained simultaneously another proof of that result;
instead of looking at the geometry of horoballs, as in our Theorem~\ref{thb}, 
they prove a very nice result on the length of elements of $\Cr_2(\bfk)$ in terms of the
generators provided by Noether-Castelnuovo theorem (the generating sets being $\PGL_3(\bfk)$
and transformations preserving a pencil of lines). Our proof applies directly to Halphen twists on 
non-rational surfaces. 

We thank J\'er\'emy Blanc and Jean-Philippe Furter, as well as Vincent Guirardel, Anne Lonjou, 
and Christian Urech for interesting discussions on this topic.

\section{Degrees and upper bounds on the distortion}\label{par:dyna-degree}
  
The following proposition shows that the degree growth may be used to control the distortion of
a birational transformation. 

\begin{pro}\label{degd}
Let $(M,H)$ be a polarized projective variety, and $f$ be a birational transformation of $M$.
\begin{enumerate}
\item If $\deg_H(f^n)$ grows exponentially, then $f$ is undistorted. 
\item If $\deg(f^n)\succeq n^\alpha$ for some $\alpha>0$, the distortion of $f$ is at most exponential.
\end{enumerate}
\end{pro}
\begin{proof}
According to Equation~\eqref{eq:submultiplicative-degrees}, the degree function is almost submultiplicative; 
replace it by $\deg_H'(f):=\deg_H(f)/C_{M,H}$ to get a submultiplicative function. 

If $S$ is a finite symmetric subset of $\Bir(M)$, and $D$ is the maximum of $\deg'_H(g)$ for $g$ in $S$, 
then $D^n$ is an upper bound for $\deg'_H$ on the ball $S^n$. 
Hence if $\deg'(f^m)\geq Cq^m$ for some constants $C>0$  and $q>1$, and if $f^m\in S^n$ we have $Cq^m\leq D^n$. Taking
logarithm, we get $m\log(q)+C\le n\log(D)$, and then $m\le \log(q)^{-1}(nlog(D)-C)$. Thus 
\[
\delta_{f,S}(n)\leq \frac{(nlog(D)-C)}{\log(q)}\preceq n
\]
and the first assertion is proved. 
Now, assume that  $\deg_H'(f^m)\ge cm^\alpha$ for some positive constants $c$ and $\alpha$.  Then 
$cm^\alpha\le D^n$, so $m\le c^{-1/\alpha}D^{n/\alpha}$. Thus $\delta_{f,S}(n)\le c^{-1/\alpha}D^{n/\alpha}\preceq\exp(n)$
and the second assertion follows. 
\end{proof}

\begin{rem} More generally, consider an increasing function $\alpha$ such that 
$\alpha(m)\leq   \log \deg'_H(f^m)$ for all $m\geq 1$. Let $\beta$ be a decreasing inverse of $\alpha$, i.e.\
a function $\beta\colon \R_+\to \R_+$ such that $\beta(\alpha(m))=m$ for all $m$.  We have 
\[
\alpha(m)\leq \log(\deg_H'(f^m))\leq n\log(D)
\]
if $f^m$ is in $S^n$, hence $\delta_{f,S}(n)\leq \beta(n\log(D))$. However, we do not know any example of 
birational transformation with  intermediate (neither exponential nor polynomially bounded) degree growth. 
See \cite{Urech} for a lower bound on the degree growth when $f\in \Aut(\A^m_\bfk)$.
\end{rem}
  
\section{Heights and distortion}\label{par:h-and-d}

In this section we study the distortion of automorphisms of $\P^m_\bfk$ in the groups
$\Aut(\P^m_\bfk)$ and  $\Cr_m(\bfk)=\Bir(\P^m_\bfk)$. 

\subsection{Distortion and monomial transformations}

Let $\bfk$ be an algebraically closed field of characteristic zero.
Here, we show that all elements of $\PGL_{m+1}(\bfk)$ are distorted in $\Cr_m(\bfk)$, and we compute their distortion rate. 

\subsubsection{Monomial transformations and distortion of semisimple automorphisms}

The group $\GL_m(\Z)$ acts by automorphisms on the $m$-dimensional multiplicative group $\Gm^m$:
if $A=[a_{i,j}]$  is in $\GL_m(\Z)$, then 
$A(x_1, \ldots, x_m)=(y_1, \ldots, y_m)$ with
\begin{equation}
y_j=\prod_i x_i^{a_{i,j}}.
\end{equation}
 The group $\Gm^m(\bfk)$ acts also on itself by translations. 
Altogether, we get an embedding of $\GL_m(\Z)\ltimes \Gm^m(\bfk)$ in $\Bir(\P^m_\bfk)$.

If $s$ is a fixed element of $\bfk^\times$, we denote by $\varphi_s\colon \Z^m\to \Gm^m$ the homomorphism defined by $\varphi_s(n_1,\dots,n_d)=(s^{n_1} ,\dots,s^{n_d} )$. This homomorphism is injective if and only if $s$ is not a root of unity.  
Its image   $\varphi_s(\Z^m)$ is normalized by the monomial group $\GL_d(\Z)$; in this way, every element $s\in \bfk^\times$
of infinite order determines an embedding of $\GL_m(\Z)\ltimes \Z^m$ into $\Bir(\P^m_\bfk)$, the image
of which is $\GL_m(\Z)\ltimes \varphi_s(\Z^m)$.
The following lemma is classical (see~\cite{LMR1,LMR2} for instance). 

\begin{lem}
For every $m\ge 2$, the abelian subgroup $\Z^m$  is exponentially distorted in $\GL_m(\Z)\ltimes \Z^m$. More precisely, 
$\vert g^n\vert \simeq \log(n)$ for every non-trivial element $g$ in the (multiplicative) abelian group $\Z^m$.
\end{lem}

For $u\in\bfk^\times$, the subgroup $\varphi_u(\Z^m)$ of $\Gm^m(\bfk)$ acts by translations on $\Gm^m(\bfk)$. This determines a  
subgroup $V_u$ of $\Cr_m(\bfk)$ acting by diagonal transformations $(x_1,\dots,x_m)\mapsto (u^{n_1}x_1,\dots,u^{n_m}x_m)$.  
By the previous lemma, the distortion of every element in $V_u$ is at least exponential in $\Cr_m(\bfk)$ (when $u$ is a root of unity, 
the distortion is infinite).

Now let $u$ be an arbitrary diagonal transformation: $u(x)=(u_1x_1, \ldots, u_mx_m)$, where $(u_i) \in\Gm^m(\bfk)$. Consider
the transformations $g_i=(x_1,\dots,x_{i-1},u_ix_i,x_{i+1},\dots, x_m)$. Then the $g_i$ pairwise commute and $u=g_1\dots g_m$. Since $g_i\in V_{u_i}$, it is at least exponentially distorted in $\GL_m(\Z)\ltimes \Gm^m(\bfk)$. Thus, $u$ is at least exponentially distorted in $\GL_m(\Z)\ltimes \Gm^m(\bfk)$. We have proved:
 
\begin{lem}\label{lem:diago-distortion}
Let $\bfk$ be a field and $m\geq 2$ be an integer. 
In $\Bir(\P^m_\bfk)$,  every linear, diagonal transformation is at least exponentially distorted. 
\end{lem}

\subsubsection{Distortion of unipotent automorphisms}

\begin{lem}\label{lem:unip-distortion}
If $U$ is a unipotent element of $\SL_{m+1}(\bfk)$, then $U$ is at least exponentially distorted in 
$\SL_{m+1}(\bfk)$, and it is at least doubly exponentially distorted in $\Bir(\P^m_\bfk)$ for $m\geq 2$. 
\end{lem}

Consequently,  the image of $U$ has finite order in every linear representation of (large enough subgroups of) the Cremona group.
Note that (in characteristic zero) this already indicates that $\Cr_1(\bfk)\subset\Cr_2(\bfk)$ is distorted in the sense that the translation 
$x\mapsto x+1$, which has exponential distortion in $\Cr_1(\bfk)\simeq\PGL_2(\bfk)$, has double exponential distortion in $\Cr_2(\bfk)$. 

\begin{proof}
Unipotent elements of $\SL_{m+1}(\bfk)$ have finite order if the characteristic of the field is positive; hence, we assume that
${\mathrm{char}}(\bfk)=0$. Consider the element 
\begin{equation}
U=\begin{pmatrix} 1 & 1\\ 0 & 1\end{pmatrix}
\end{equation}
of $\SL_2(\bfk)$. Let $A\in \SL_2(\bfk)$ be the diagonal matrix with coefficients $2$ and $1/2$ on the diagonal: 
$A^nUA^{-n}=U^{4^n}$ and $U$ is exponentially distorted  in the subgroup of $\SL_2(\bfk)$ generated by $U$ and $A$.
 Similarly, consider a unipotent matrix $U_{i,j}= \Id+ E_{i,j}$, where $E_{i,j}$ is the  $(m+1)\times (m+1)$ matrix with only
 one non-zero coefficient, namely  $e_{i,j}=1$; 
then $U_{i,j}$ is exponentially distorted in $\SL_{m+1}(\bfk)$: there is a diagonal matrix
$A$ such that $\vert U_{i,j}^n\vert \simeq \log(n)$ in the group $\langle U_{i,j}, A\rangle$, for all $n\geq 1$. 
This implies that unipotent matrices are exponentially distorted in $\SL_{m+1}(\bfk)$.

As a second step, consider a $3\times 3$ Jordan block and its iterates:
\begin{equation}
U= \begin{pmatrix} 1 & 1 & 0 \\ 0 & 1 & 1 \\ 0 & 0 & 1 \end{pmatrix}, \quad U^n= \begin{pmatrix} 1 & n & n(n-1)/2 \\ 0 & 1 & n \\ 0 & 0 & 1 \end{pmatrix}.
\end{equation}
We want to prove that $U$ is doubly exponentially distorted in $\Cr_2(\bfk)$. 
Take iterates $U^{K^n}$ for some integer $K>1$. Then, conjugating by $A^n$, and multiplying by $B$, with
\begin{equation}
A= \begin{pmatrix} 1 & 0 & 0 \\ 0 & K & 0 \\ 0 & 0 & 1 \end{pmatrix}, \quad B= \begin{pmatrix} 1 & 0 & 0 \\ 0 & 1 & -1 \\ 0 & 0 & 1 \end{pmatrix}, \quad C=\begin{pmatrix} K & 0 & 0 \\ 0 & 1 & 0 \\ 0 & 0 & 1 \end{pmatrix}
\end{equation}
we get a new matrix $BA^{-n}U^{K^n}A^n=[v_{i,j}(n)]$ which is upper triangular; its coefficients are equal to $1$ on the diagonal, $v_{1,2}=K^n$,  $v_{1,3}=K^n(K^n-1)/2$ 
and $v_{2,3}=0$. Conjugating with $C^n$ changes $v_{1,2}$ into $v_{1,2}'=1$ and $v_{1,3}$ into $v_{1,3}'=(K^n-1)/2$. 
Multiplying by the unipotent matrix $D=\Id-E_{1,2}+1/2E_{1,3}$ changes 
$v_{1,2}'$ into $0$ and $v'_{1,3}$ into $K^n$. One more conjugacy by $C^n$ gives a matrix $E$ with constant coefficients. Thus $U^{K^n}$ 
is a word of finite length (independent of $n$) in $A^n$, $C^n$, and a fixed, finite number of unipotent matrices ($B$, $D$, $E$). 
Since $A$ and $C$ are diagonal matrices, they satisfy $\vert A^n\vert\sim \log(n)$ and $\vert C^n\vert\sim \log(n)$ in some finitely generated subgroup of $\Cr_2(\bfk)$. Thus, $U$ is doubly exponentially distorted. 

This argument and a recursion starting at $m=2$ proves the general result. 
\end{proof}

\subsubsection{Distortion of linear projective transformations}\label{par:disto-autom-Pm}

Every $A\in \PGL_{m+1}(\bfk)$ is the product
of a semisimple element $S_A$ with a unipotent element $U_A$ such that $S_A$ and $U_A$ 
commute. When $\bfk$ is algebraically closed, $S_A$ is diagonalizable. 
By Lemmas~\ref{lem:diago-distortion} and~\ref{lem:unip-distortion}, $A$ is at least
exponentially distorted (resp. doubly exponentially distorted if $S_A$ has finite order).

\subsection{Heights and upper bounds}

\begin{thm}\label{thm:dbl-expo}
Let $\bfk$ be an algebraically closed field of characteristic zero. Let $A$ be an element of $\Aut(\P^m_\bfk)$
given by a matrix in $\SL_{m+1}(\bfk)$ of infinite order. Then, its distortion in the Cremona group $\Bir(\P^m_\bfk)$ 
is doubly exponential if the matrix is virtually unipotent, and simply exponential otherwise. \end{thm}

To prove this result, we use basic properties of heights of polynomial functions. 
We start with a 
proof of this theorem when $\bfk=\Qbar$ is an algebraic closure of the field of rational 
number; the general case is obtained by a specialization argument. 

\subsubsection{Heights of polynomial functions}\label{par:heights-poly}

Let $\bfK$ be a finite extension of $\Q$, and let $M_\bfK$ be the set of places of $\bfK$; to each place, 
we associate a unique absolute value $\vert \cdot \vert_v$ on $K$, normalized as follows (see \cite{BG}, \S 1.4).
First, for each prime number $p$, the $p$-adic absolute value on $\Q$ satisfies  $\vert p\vert_p=1/p$, and $\vert \cdot \vert_\infty$
 is the standard absolute value. Then, 
if $v\in M_\bfK$ is a place that divides $p$, with $p$ prime or $\infty$, then 
\begin{equation}
\vert x\vert_v=\vert \Norm_{\bfK/\Q}(x) \vert_p^{1/[\bfK:\Q] }
\end{equation}
for every $x\in \bfK$. With such a choice, the product formula reads 
\begin{equation}
\sum_{v\in M_\bfK}\log \vert x\vert_v =0
\end{equation}
for every $x\in \bfK\smallsetminus\{0\}$.

Let $m$ be a natural integer. If $f(\bfx)=\sum_I a_I \bfx^I$ is a polynomial function in the variables 
$\bfx=(x_0, \ldots, x_m)$, 
with $a_I\in \bfK$ for each multi-indice $I=(i_0, \ldots, i_m)$,   we set 
\begin{equation}
\vert f \vert_v = \max_{I} \vert a_I\vert_v
\end{equation}
for every place $v\in M_\bfK$. If $f\neq 0$, we define its {\bf{height}} $h(f)$ by
\begin{equation}
h(f)= \sum_{v\in M_\bfK} \log \vert f\vert_v.
\end{equation}
If $\hat{f}=(f_0, \ldots, f_m)$ is an endomorphism of $\A^{m+1}_\bfK$, the height $h(\hat{f})$ 
is the maximum of the heights $h(f_i)$, and $\vert \hat{f}\vert_v$ is the maximum of the $\vert f_i\vert_v$.
(Note that the affine coordinates system $\bfx$ is implicitly fixed.)  

\begin{rem}\label{rem:heights} Let $f$ and $g$ be non-zero  elements of $\bfK[x_0, \ldots, x_m]$.

(1).-- The product formula implies that $h(af)=h(f)$, $\forall a\in \bfK\smallsetminus\{0\}$. 

(2).-- From this, we see that $h(f)\geq 0$ for all $f\in\bfK[x_0,\ldots,x_m]\smallsetminus\{0\}$. Indeed, one can multiply
$f$ by the inverse of a coefficient $a_I\neq 0$ without changing the value of its height; then, one of the coefficients
is equal to $1$ and $\vert f\vert_v\geq 1$ for all $v\in M_\bfK$. 

(3).-- The Gauss Lemma says that $\vert fg\vert_v=\vert f\vert_v \vert g\vert_v$ when  $v$ is not archimedean.
This multiplicativity property fails for places at infinity. 

(4).-- If ${\mathbf{L}}$ is an extension of $\bfK$, then the height of $f\in \bfK[x_0, \ldots, x_m]$ is the
same as its height as an element of ${\mathbf{L}}[x_0, \ldots, x_m]$ (see \cite{BG}, Lemma 1.3.7).
Thus, the height is well defined on $\Qbar[x_0, \ldots, x_m]$. 
\end{rem}

\begin{thm}[see \cite{BG}, 1.6.13]\label{thm:Gelfond} Let $f_1$, $\ldots$, $f_s$ be non-zero elements of $\Qbar[x_0, \ldots, x_m]$, 
and let $f$ be their product $f_1\cdots f_s$. Let $\Delta(f)$ be the sum of the partial degrees of $f$ with respect
to each of the variables $x_i$. Then 
\[
-\Delta(f) \log(2) +\sum_{i=1}^sh(f_i) \; \leq \; h(f) \; \leq\; \Delta(f) \log(2) +\sum_{i=1}^sh(f_i).
\] 
\end{thm}
If $\deg(f)$ denotes the degree of $f$, then $\Delta(f)\leq (m+1)\deg(f)$. For $s=2$ we get 
\begin{equation}
 h(f_1) \; \leq \; h(f) - h(f_2) +(m+1)\log(2) \deg(f).
\end{equation}

\subsubsection{Heights of birational transformations}\label{par:Heights-Bir}
Consider a birational transformation $f\colon \P^m_\Qbar\dasharrow \P^m_\Qbar$, and write it in homogeneous coordinates
\begin{equation}
f[x_0:\ldots :x_m]=[f_0:\ldots : f_m]
\end{equation}
where the $f_i\in \Qbar[x_0, \ldots, x_m]$ are homogeneous polynomial functions of the same degree $d$
with no common factor of positive degree. Then, $d$ is the degree of $f$ (see \S~\ref{par:degrees}), and the $f_i$ are uniquely determined modulo 
multiplication by a common constant $a\in \Qbar\setminus\{0 \}$. Thus, Remark~\ref{rem:heights}(1) shows that the real 
number 
\begin{equation}
h(f)=\max_{i} h(f_i)
\end{equation}
is well defined. This number $h(f)$ is, by definition, the {\bf{height}} of the birational transformation $f$. 
It coincides with the height of the lift of $f$ as the endomorphism $\hat{f}=(f_0, \ldots, f_{m+1})$ of $\A^{m+1}_\bfk$ (see \S~\ref{par:heights-poly}).

\subsubsection{Growth of heights under composition}\label{par:Heights-growth1}

Let $S=\{ f^1, \ldots, f^s\}$ be a finite symmetric set of birational transformations of $\P^m_\Qbar$; the symmetry means
that $f\in S$ if and only if $f^{-1}\in S$. Consider the homomorphism from the 
free group $\F_s=\langle a_1, \ldots, a_s\vert \emptyset\rangle$ to $\Bir(\P^m_\Qbar)$ defined by mapping each generator $a_j$ to $f^j$.
Then, to every reduced word $w_\ell(a_1, \ldots, a_s)$ of length $\ell$ in the generators $a_j$ corresponds an element 
\begin{equation}
w_l(S)=w_l(f^1, \ldots, f^s)
\end{equation}
of the Cremona group $\Bir(\P^m_\Qbar)$. 

For each $f^i\in S$, we fix a system of homogeneous polynomials $f^i_j\in \Qbar[x_0:\ldots :x_m]$ defining $f$, as in \S~\ref{par:Heights-Bir}: 
$f^i=[f^i_0:\ldots: f^i_m]$ and the $f^i_j$ have degree $d_i=\deg(f^i)$. Moreover, we choose the $f^i_j$ so that for every $i$ at least one
of the coefficients of the $f^i_j$ is equal to $1$. Once the $f^i_j$ have been fixed, we have a canonical lift of 
each $f^i$ to a homogeneous endomorphism $\hat{f^i}$ of $\A^{m+1}_\Qbar$, given by 
\begin{equation}
\hat{f^i}(x_0, \ldots, x_m)=(f^i_0, \ldots, f^i_m).
\end{equation}
Thus, every reduced word $w_\ell$ of length $\ell$ in $\F_s$ determines also an endomorphism 
$\hat{w_\ell}(S)= w_\ell(\hat{f^1}, \ldots, \hat{f^s})$ of the affine space. 

Let $d_S$ be the maximum of $\{2, d_1, \ldots, d_s\}$, so that $d_S\geq 2$. Then, the degree of the endomorphism $\hat{w_\ell}(S)$ 
is at most $d_S^\ell$. 

Let $\bfK$ be the finite extension of $\Q$ which is generated by all the coefficients $a^i_{j,I}$ of the polynomial functions $f^i_j=\sum a^i_{j, I}\bfx^I$. 
We shall say that a place $v\in M_\bfK$ is {\bf{active}} if $\vert a^i _{j, I}\vert_v>1$ for at least one of these
coefficients; the set of active places is finite, because there are only finitely many coefficients. 
For each place $v\in M_\bfK$, we set 
\begin{equation}
M(v)= \max{\vert a^i_{j,I}\vert_v}= \max \vert \hat{f^i}\vert_v,
\end{equation}
the maximum of the absolute values of the coefficients; our normalization implies that $M(v)\geq 1$ and $M(v)=1$ if and only if $v$ is not active. 

\begin{lem}\label{lem:growth-height1}
Let $v$ be a non-archimedean place. If $w_\ell\in \F_s$ is a reduced word of length $\ell$, then 
\[
\log \vert \hat{w_\ell}(S)\vert_v \leq \log(M(v)) d_S^{\ell}.
\]
Thus, if $v$ is not active, then $\log\vert \hat{w_\ell}(S)\vert_v=0$.
\end{lem}

\begin{proof}
Set $d=d_S$. Write $\hat{w_\ell}(S)$ as a composition $\hat{g}^\ell\circ \cdots \circ \hat{g}^1$, where each $\hat{g}^k$ is one of the $\hat{f}^i$
(here we use that $S$ is symmetric). By definition, $\vert \hat{g}^1 \vert_v\leq M(v)$. Then, assume that $\vert \hat{g}^{k-1}\circ \cdots \circ \hat{g}^1\vert_v\leq M(v)^{1+d+\cdots + d^{k-2}}$ for some integer $2\leq k\leq \ell$. Write $\hat{g}^{k-1}\circ \cdots \circ \hat{g}^1=(u_0, \ldots, u_m)$ for some
homogeneous polynomials $u_j$.   The Gauss lemma (see Remark~\ref{rem:heights}) says that
\begin{equation}
\vert u_0^{i_0} \cdots u_m^{i_m}\vert_v=\vert u_0\vert_v^{i_m} \cdots \vert u_m\vert_v^{i_0}\leq (M(v)^{1+d+\cdots + d^{k-2}})^d
\end{equation}
for every multi-index $I=(i_0, \ldots, i_m)$ of length $\sum i_j\leq d$.
The endomorphism $\hat{g}^k$ has degree $\leq d$, and the absolute values of its coefficients
are bounded by $M(v)$, hence
\begin{equation}
\vert \hat{g}^{k}\circ \cdots \circ \hat{g}^1\vert_v\leq M(v)^{1+d+\cdots + d^{k-1}}.
\end{equation}
By recursion, this upper bound holds up to $k= \ell$. For $k=\ell$ we obtain the estimate $\log \vert \hat{w_\ell}(S)\vert_v \leq \log(M(v)) d^{\ell}$
because $1+d+\cdots +d^{\ell-1}\leq d^\ell$.
\end{proof}

\begin{lem}\label{lem:growth-height2}
Let $v$ be an archimedean place. If $w_\ell\in \F_s$ is a reduced word of length $\ell$, then 
\[
\log \vert \hat{w_\ell}(S)\vert_v \leq  2\log(M(v)) d_S^\ell + \log(md_S^m) d_S^{2\ell}.
\]
\end{lem}

\begin{proof}
Consider a monomial $\bfx^I= x_0^{i_0}\cdots x_m^{i_m}$ of degree $\leq d$. Let $u_0$, $\ldots$, $u_m$
be homogeneous polynomials of degree $\leq D$ with $D\geq 2$ and with all coefficients satisfying $\vert c\vert_v\leq C$. Note that
the space of homogeneous polynomials of degree $D$ in $m$ variables has dimension $\binom{D+m}{m}$.Then
\begin{equation}
\vert u_0^{i_0} \cdots u_m^{i_m}\vert_v\leq \binom{D+m}{m}^dC^d\leq (D+m)^{md} C^d \leq (mD^m)^dC^d
\end{equation}
Indeed, every coefficient in the product $u_0^{i_0} \cdots u_m^{i_m}$ is obtained as a sum of at most $\binom{D+m}{m}^d$
terms, each of which is a product of at most $d$ coefficients of the $u_j$.  

Then, to estimate the absolute values of the coefficients of $\hat{w_\ell}(S)$, we proceed by recursion as
in the proof of Lemma~\ref{lem:growth-height1}. Set $B=md_S^m$.  
For a composition $\hat{g}^k\circ \cdots \circ \hat{g}_1$ of length $k$ we obtain 
\begin{equation}
 \vert \hat{g}^k\circ \cdots \circ \hat{g}^1\vert_v\leq B^{(k-1)d_S^{k-1}} M(v)^{2d_S^{k-1}}.
\end{equation}
The conclusion follows from $\ell d_S^\ell \leq d_S^{2\ell}$. 
\end{proof}

Putting these lemmas together, we get 
\begin{equation}
h(\hat{w_\ell}(S))\leq \sum_{v \; \text{active}} 2\log(M(v)) d_S^{\ell} +\sum_{v\vert \infty} \log(md_S^m) d_S^{2\ell}
\end{equation}
This inequality concerns the height of the endomorphism $\hat{w_\ell}(S)$; to obtain the 
birational transformation $w_\ell(S)$, we might need to divide by a common factor $q(x_0, \ldots, x_m)$. Since the 
degree of $\hat{w_\ell}(S)$ is no more than $d_S^\ell$, Theorem~\ref{thm:Gelfond} provides the upper bound
\begin{equation}
h(w_\ell(S))\leq  \left( (m+1) \log(2)+\sum_{v \; \text{active}}  \log(M(v))+\sum_{v\vert \infty} \log(md_S^m) \right) d_S^{2\ell}.
\end{equation}
This proves the following proposition. 

\begin{pro}\label{pro:Height-expo-bound}
Let  $\F_s=\langle a_1, \ldots, a_s\vert \emptyset\rangle$ be a free group of rank  $s\geq 1$.
For every  homomorphism $\rho\colon\F_s\to \Bir(\P^m_\Qbar)$,  there exist two constants $C_m(\rho)$ and $d(\rho)\geq 1$ such that 
$h(\rho(w))\leq C_m(\rho) d(\rho)^{\vert w\vert}$ for every $w\in \F_s$, where $\vert w\vert$ is the length of $w$ as a reduced word in 
the generators $a_i$. 
\end{pro}

\subsection{Proof of Theorem~\ref{thm:dbl-expo}}

We may now prove Theorem~\ref{thm:dbl-expo}. When $\bfk=\Qbar$, this result is a direct
corollary of Proposition~\ref{pro:Height-expo-bound} and Section~\ref{par:disto-autom-Pm}; we start with this case and then treat the general case
via a specialization argument.

\subsubsection{Number fields}
Let $A$ be an element of $\SL_{m+1}(\Qbar)$ of infinite order. 
After conjugation, we may assume $A$ to be upper triangular. 
First, suppose that  $A$ is virtually unipotent (all its eigenvalues are roots of unity). Then $h(A^n)$ grows 
like $\tau \log(n)$ as $n$ goes to $+\infty$. Thus, if $A^n$ is a word of length 
$\ell(n)$ in some fixed, finitely generated subgroup of $\Bir(\P^m_\Qbar)$, Proposition~\ref{pro:Height-expo-bound}
shows that 
\begin{equation}
\tau\log(n)\le Cd^{\ell(n)}
\end{equation}
for some positive constants $C$ and $d> 1$. Thus, $A$ is at most doubly exponentially distorted; from 
Lemma~\ref{lem:unip-distortion}, it is exactly doubly exponentially distorted.
Now, suppose that an eigenvalue $\alpha$ of $A$ is not a root of unity. Kronecker's lemma provides
a place $v\in M_{\Q(\alpha)}$ for which $\vert \alpha \vert_v>1$ (see~\cite{BG}, Thm. 1.5.9). Thus, $h(A^n)$ grows like $\tau n$
for some positive constant $\tau$ as $n$ goes to $+\infty$ (see Remark~\ref{rem:heights}(2)), and $A$ is at most exponentially distorted
in $\Bir(\P^m_\Qbar)$. From Section~\ref{par:disto-autom-Pm}, we obtain Theorem~\ref{thm:dbl-expo} when $\bfk= \Qbar$.

\subsubsection{Fields of characteristic zero}

Let $\bfk$ be an algebraically closed field of characteristic zero and let $A$ be an element of $\SL_{m+1}(\bfk)$. 
Let $S=\{ f^1, \ldots, f^m\}$ be a finite symmetric subset of $\Bir(\P^m_\bfk)$ such that the group generated by $S$ 
contains $A$. For each $n$, denote by $\ell(n)$ the length of $A^n$ as a reduced word in the $f^i$. 

Write each $f^i$ in homogeneous coordinates $f^i=[f^i_0:\ldots : f^i_m]$, as in Section~\ref{par:Heights-growth1}; and
denote by ${\mathcal{C}}$ the set of  coefficients of the matrix $A$ and of the polynomial functions $f^i_j=\sum a^i_{j, I}\bfx^I$. This is a finite subset
of $\bfk$, generating a finite extension $\bfK$ of $\Q$. This finite extension is an algebraic extension of a purely 
trancendental extension $\Q(t_1, \ldots, t_r)$, where $r$ is the transcendental degree of $\bfK$ over $\Q$. 
Then, the elements of ${\mathcal{C}}$ are algebraic functions with coefficients in $\Qbar$ (such as $(2t_1t_3^2-1)^{1/3}+t_2^5$);
the ring of functions generated by ${\mathcal{C}}$ (over $\Qbar$) may be viewed as the ring of functions of some algebraic variety $V_{\mathcal{C}}$
(defined over $\Qbar$).

If $u$ is a point of $V_{\mathcal{C}}(\Qbar)$ and $c\in {\mathcal{C}}$ is one of the coefficients, we may evaluate
$c$ at $u$ to obtain an algebraic number $c(u)$. Similarly, we may evaluate, or specialize, $A$ and the $f^i$ at $u$.
This gives an element $A_u$ in $\SL_{m+1}(\Qbar)$ (the determinant is $1$), and rational transformations $f^i_u$ of $\P^m_{\Qbar}$.
For some values of $u$, $f^i_u$ may be degenerate, identically equal to $[0:\ldots : 0]$; but for $u$ in a dense, Zariski open subset
of $V_{\mathcal{C}}$, the $f^i_u$ are birational transformations of degree $\deg(f^i_u)=\deg(f^i)$. Pick such a point 
$u\in V_{\mathcal{C}}(\Qbar)$. If $A^n$ is a word of length $\ell(n)$ in the $f^i$, then $A^n_u$ is a word of the
same length in the $f^i_u$. From the previous section we deduce that $A$ is at most doubly exponentially distorted. 
Moreover, if one of the eigenvalues $\alpha \in \bfk$ of $A$ is not a root of unity, we may add $\alpha$ to the set
${\mathcal{C}}$ and then choose the point $u$
such that $\alpha(u)$ is not a root of unity either. Then, $A_u$ and thus $A$ is at most exponentially distorted.  This concludes 
the proof of Theorem~\ref{thm:dbl-expo}.

\section{Non-distortion}\label{prundi}

In this section, we prove Theorem~\ref{thb}, which provides an upper
bound for the distortion of parabolic isometries in certain groups of isometries of hyperbolic
spaces.

\subsection{Hyperbolic spaces and parabolic isometries}\label{par:Hyperb-space}
\subsubsection{Hyperbolic spaces} Let ${\mathcal{H}}$ be a real Hilbert space of dimension $m+1$ ($m$ can be infinite). 
Fix a unit vector $\bfe_0$ of ${\mathcal{H}}$ and a Hilbert basis $(\bfe_i)_{i\in I}$ of the orthogonal complement of $\bfe_0$. 
Define a new scalar product on ${\mathcal{H}}$ by 
\begin{equation}
\langle u \vert u' \rangle = a_0a'_0-\sum_{i\in I} a_i a'_i
\end{equation}
for every pair $u=a_0\bfe_0+\sum_i a_i \bfe_i$, $u'= a'_0\bfe_0+\sum_i a'_i\bfe_i$ of vectors. 
Define $\HH_{m}$ to be the connected component of the hyperboloid 
$\{u\in {\mathcal{H}}\vert \; \langle u \vert u \rangle=1\}$
that contains $\bfe_0$, and let $\dist$ be the distance
on $\HH_m$ defined by (see \cite{Benedetti-Petronio})
\begin{equation}
\cosh( \dist(u, u') ) = \langle u \vert u'\rangle.
\end{equation}
The metric space $(\HH_m, \dist)$ is a model of the hyperbolic space of dimension $m$ (see \cite{Benedetti-Petronio}). 
The projection of $\HH_m$ into the projective space $\P(\mathcal{H})$ is one-to-one onto its image. 
 In what follows, $\HH_m$ is identified with its image in $\P(\mathcal{H})$ and its boundary is denoted by
$\partial \HH_m$; hence, boundary points correspond to isotropic lines in the space $\mathcal{H}$ for the scalar product $\langle \cdot \vert \cdot\rangle$.

\subsubsection{Hyperbolic plane} A useful model for $\HH_2$ is the Poincar\'e model: $\HH_2$ is identified to the upper
half-plane $\{z\in\C;\Ima(z)>0\}$, with its Riemannian metric given by $ds^2=(x^2+y^2)/y^2$. Its group of orientation preserving isometries
coincides with $\PSL_2(\R)$, acting by linear fractional transformations. The distance between two points $z_1$ and $z_2$ satisfies
\begin{equation}\label{eq:dist-h2}
\sinh \left( \frac{1}{2} \dist_{\HH_2}(z_1,z_2)\right) = \frac{\vert z_1-z_2\vert }{2(\Ima(z_1)\Ima(z_2))^{1/2}}.
\end{equation}

\subsubsection{Isometries} 
 Denote by $\sfO_{1,m}(\R)$ the group of linear transformations of $\mathcal{H}$ 
preserving the scalar product $\langle \cdot \vert \cdot \rangle$. The group of isometries $\Isom(\HH_m)$
 coincides with the index $2$ subgroup $\sfO^+_{1,m}(\R)$ of $\sfO({\mathcal{H}})$
 that preserves the chosen sheet $\HH_m$ of the hyperboloid $\{u\in {\mathcal{H}}\, \vert \, \langle u \vert u \rangle=1\}$. 
 This group acts transitively on $\HH_m$, and on its unit tangent bundle.

If $h\in \sfO^+_{1,m}(\R)$ is an isometry of $\HH_m$ and $v\in {\mathcal{H}}$ is an eigenvector of $h$ with eigenvalue 
$\lambda$, then either $\vert \lambda \vert =1$ or
$v$ is isotropic. Moreover, since $\HH_m$ is homeomorphic to a ball, $h$ has at least one eigenvector $v$ in $\HH_m\cup \partial \HH_m$.
Thus, there are three types of isometries~\cite{Burger-Iozzi-Monod}: 

\vspace{0.1cm}
\noindent {$\;$(1)$\;$} An isometry $h$ is {\bf{elliptic}} if and only if it fixes a point $u$ in $\HH_m$. Since $\langle \cdot \vert \cdot  \rangle$ is negative
definite on  the orthogonal complement $u^\perp$, the linear transformation $h$
fixes pointwise the line $\R u$  and acts by rotation on $u^\perp$ with respect to~$\langle \cdot \vert \cdot  \rangle$.

\vspace{0.1cm}
\noindent {$\;$(2)$\;$}  An isometry $h$ is {\bf{parabolic}} if it is not elliptic and fixes a vector $v$ in the isotropic cone. The line $\R v$ is uniquely determined 
by the parabolic isometry~$h$. If $z$ is a point of $\HH_m$, there is an increasing sequence of integers $m_i$ such that $h^{m_i}(z)$ converges
towards the boundary point $\xi$ determined by $v$.

\vspace{0.1cm}
\noindent {$\;$(3)$\;$} An isometry $h$ is {\bf{loxodromic}} if and only if $h$ has an eigenvector $v^+_h$ with eigenvalue $\lambda>1$. Such an eigenvector is unique up to scalar multiplication, and there is another, unique,  isotropic eigenline $\R v^-_h$ corresponding to an eigenvalue $<1$; this eigenvalue is equal to $1/\lambda$.  On the orthogonal complement of $\R v^+_h\oplus \R v^-_h$, $h$ acts as a rotation with 
respect to $\langle  \cdot \vert \cdot  \rangle$. 
The boundary points determined by $v^+_h$ and $v^-_h$ are the two fixed points of $h$ in $\HH_\infty\cup \partial\HH_\infty$: 
the first one is an attracting fixed point, the second is repelling. 
Moreover,  $h\in \Isom(\HH_\infty)$ is 
loxodromic if and only if its {\bf{translation length}}
\begin{equation}
L(h)=\inf \{ \dist(x,h(x)) \; \vert \; x\in \HH_\infty  \}
\end{equation}
is positive. In that case, $\lambda=\exp(L(h))$ is the largest eigenvalue of $h$ and $\dist(x,h^n(x))$ grows like $nL(h)$ as $n$ goes to $+\infty$ for
every point $x$ in $\HH_m$. 

\vspace{0.1cm}

When $h$ is elliptic or parabolic, the translation length vanishes (there is a point $u$ in $\HH_m$ with $L(h)=\dist(u,h(u))$ if $h$ is elliptic, but
no such point exists if $h$ is parabolic).

\subsubsection{Horoballs}\label{par:horoballs}

Let $\xi$ be a boundary point of $\HH_m$, and let $\epsilon$ be a positive real number. 
The {\bf{horoball}} $H_\xi(\epsilon)$ in $\HH_\infty$ is the subset
\[
H_\xi(\epsilon)=\{v\in \HH_m  \; ; \; 0 < \langle v\vert \xi\rangle < \epsilon\}.
\]
It is a limit of balls with centers converging to the boundary point $\xi$. An isometry $h$ fixing the boundary point $\xi$ maps 
 $H_\xi(\epsilon)$ to  $H_\xi(e^{L(h)} \epsilon)$.

\subsection{Distortion estimate}
Our goal is to prove the following theorem. 

\begin{thm}\label{thb} Let $m$ be any (possibly infinite) cardinal. Let $G$ be a subgroup of $\Iso(\HH_m)$. Let $f$ be a parabolic  element of $G$, and let $\xi\in \partial \HH_m$ be the fixed point of~$f$.
Suppose that the following two properties are satisfied.
\begin{itemize}
\item[(i)] There are positive constants $C$ and $C'>0$ and a point $x_0\in \HH_m$ such that   
\[
\dist(f^n(x_0),x_0)\ge C\log n - C'
\]
for all large enough values of $n$.

\item[(ii)] There exists a horoball $B$ centered at $\xi$ such that for every $g\in G$ either $gB=B$ or $gB\cap B=\emptyset$.
\end{itemize}
Then $f$ is at most $n^{2/C}$-distorted in $G$. In particular $C\le 2$ and if $C=2$ then $f$ is undistorted in $G$.
\end{thm}

When looking at the Cremona group $\Cr_2(\bfk)$, 
we shall see examples of isometries in $\HH_\infty$ with $\dist(f^n(x_0),x_0)\sim C\log n$ for $C=1$ or $C=2$. 

\subsubsection{Complements of horoballs}

Let $X$ be a metric space. Let $W$ be a subset of $X$. Let $\dist_{W^c}$ (or $d_X$ when $W=\emptyset$) 
be the induced intrinsic distance on the complement of $W$; namely, for $x$ and $y$ in $W^c$, we have
 $\dist_{W^c}(x,y)=\sup_{\varepsilon>0}d_{W^c,\varepsilon}(x,y)$ with 
\[
d_{W^c,\varepsilon}(x,y)=\inf\left\{ \sum_{i=0}^{n-1}d(x_i,x_{i+1}):n\ge 0, x_0=x,x_n=y,\sup_id(x_i,x_{i+1})\le\varepsilon \right\}
\]
for points $x_i$ which are all in $X\setminus W$. It is a distance as soon as it does not take the $\infty$ value. 
In the cases we shall consider, $X$ will be the hyperbolic space $\HH_m$, $W$ will be a union of horoballs, and 
$W^c$ will be path connected. In that case, $\dist_{W^c}(x,y)$ is the infimum of the length of paths connecting
$x$ to $y$ within $W^c$. 
 
Similarly, if $Y$ is a  subset of $X$, we denote by $\dist_{Y}$ the induced intrinsic distance on $Y$ (hence, $\dist_{Y}=\dist_{(X\setminus Y)^c}$). 
 
\begin{lem}\label{compdi}
Let $B\subset \HH_m$ be an open horoball, with boundary $\partial B$.
Let $x$ and $ y$ be points on the horosphere $\partial B$.  
Then 
\[
\dist_{B^c}(x,y)= \dist_{\partial B}(x,y)=2\sinh(\dist(x,y)/2).
\]
\end{lem}
\begin{proof}
The statement being trivial when $x=y$, we assume $x\neq y$ in what follows. 
Let $\xi$ be the center at infinity of $B$. Then $x$, $y$, and the boundary point $\xi$ are contained in a 
unique geodesic plane $P$. Since the projection of $\HH_m$ onto $P$ is a $1$-Lipschitz map, we have $\dist_{\partial B}(x,y)=\dist_{\partial B\cap P}(x,y)$
and $\dist_{B^c}(x,y)=\dist_{B^c\cap P}(x,y)$. Hence, we can replace $\HH_m$ by the $2$-dimensional hyperbolic space $P\simeq \HH_2$.  
To conclude, we use the Poincar\'e half-plane model of $\HH_2$. 
There is an isometry $P\to  \HH_2$ mapping the horosphere $\partial B\cap P$
to  the line $i+\R$, the points $x$ and $y$ to $i+t$ and $i+t'$, and $\xi$ to $\infty$.
If $\gamma(s)=x(s)+iy(s)$, $s\in [a,b]$, is a path in $P\cap B^c$ that connects $x$ to $y$, its length satisfies
\[
{\mathrm{length}}(\gamma) = \int_a^b\frac{(x'(s)^2+y'(s)^2)^{1/2}}{y(s)}ds \leq \int_a^b \vert x'(s) \vert ds
\]
because $y\leq 1$ in $B^c\cap P$.
Thus, the geodesic segment from $x$ to $y$ for $\dist_{B^c}$ (resp. $\dist_\partial B$) is the euclidean segment $\gamma(s)=i+s$, with $s\in [t,t']$, and
\[
\dist_{B^c}(x,y)= \dist_{\partial B}(x,y)=\vert t-t'\vert.
\]
We conclude with Formula \eqref{eq:dist-h2}, that gives $\sinh(\dist_{\HH_2}(x,y)/2)=\vert t-t'\vert /2$. 
\end{proof}

\begin{lem}\label{diuns}
 Let $(B_i)$ be a family of open horoballs in $\HH_m$ with pairwise disjoint closures and let $Q=\bigcup B_i$. Then 
\begin{itemize}
\item[(1)] $\dist_{Q^c}$ is a distance on $\HH_m\smallsetminus Q$;
\item[(2)] for every index $i$ and every pair of points $(x,y)$ on the boundary of $\partial B_i$, we have $\dist_{Q^c}(x,y)=\dist_{\partial B_i}(x,y)$.
\end{itemize}
\end{lem}
\begin{proof}
Let $(x,y)$ be a pair of points in $Q^c$. Consider the unique geodesic segment of $\HH_m$ that joins $x$ to $y$. 
Denote by $[u_j,u'_j]$ the intersection of this segment with $B_j$. Let $C$ be a positive constant such that $2\sinh(s/2)\le Cs$ 
for all $s\in [0,\dist(x,y)]$ (such a constant depends on $\dist(x,y)$). From Lemma \ref{compdi} we obtain 
\begin{eqnarray}
\dist_{Q^c}(x,y) & \leq & \dist(x,y)+\sum_j \dist_{\partial B_j}(u_j,u'_j) \\
 & \leq & \dist(x,y)+\sum_j C\dist(u_j,u'_j)\\ 
  & \leq & (1+C)\dist (x,y)
\end{eqnarray}
Thus, $\dist_{Q^c}(x,y)$ is finite: this proves (1).

For (2), note that $Q^c\subset B_i^c$ and Lemma~\ref{compdi} imply $\dist_{Q^c}(x,y)\geq \dist_{B_i^c}(x,y)=\dist_{\partial B_i}(x,y)$, and that $\dist_{\partial B_i}(x,y)\geq \dist_{Q^c}(x,y)$ because $\partial B_i\subset \overline{Q^c}$. \end{proof}
 
\subsubsection{Proof of Theorem~\ref{thb}} 
Changing $B$ in a smaller horoball, we  can suppose that $B$ is open and that $g\bar{B}\cap \bar{B}$ is empty for all $g\in G$ with $gB\neq B$.
Let $Q$ be the union of the horoballs $gB$ for $g\in G$. 
Let $x_1$ be a point on the horosphere $\partial B$. Let $D>1$ satisfy $2\log(D)=C'+2\dist(x_1,x_0)$. 
From the first hypothesis, we know that 
\[
\dist(f^n(x_1),x_1)\ge C\log n - 2\log(D)
\]
for all sufficiently large values of $n$.
By Lemmas \ref{compdi} and \ref{diuns}, we get 
\begin{eqnarray*}
\dist_{Q^c}(x_1,f^n(x_1)) & = & 2\sinh(\dist(x_1,f^n(x_1))/2)\\
 & \ge & 2\sinh(C\log(n)/2-\log(D))\\
& \ge &  D^{-1}n^{C/2}-Dn^{-C/2}
\end{eqnarray*}
for large enough $n$.
Let us now estimate the distortion of $f$ in $G$. Let $S$ be a finite symmetric subset of $G$ and let $D_S$ 
be the maximum of the distances $\dist_{Q^c}(g(x_1),x_1)$ for $g$ in $S$.
Suppose that $f^n=g_1\circ g_2\circ \cdots \circ g_\ell$ is a composition of $\ell$ elements $g_i\in S$. 
The group generated by the $g_i$ acts by isometries on $Q^c$ for the distance $\dist_{Q^c}$. 
Thus, 
\begin{eqnarray*}
\dist_{Q^c}(f^n (x_1),x_1) & = & \dist_{Q^c}(g_1\circ \cdots \circ g_\ell(x_1),x_1) \\
& \leq & \dist_{Q^c}(g_1\circ \cdots \circ g_{\ell}(x_1), g_1\circ \cdots \circ g_{\ell-1}(x_1)) \\
& & + \dist_{Q^c}(g_1\circ \cdots \circ g_{\ell-1}(x_1),x_1) \\
& \leq & \dist_{Q^c}(g_{\ell}(x_1),x_1) + \dist_{Q^c}(g_1\circ \cdots \circ g_{\ell-1}(x_1),x_1) \\
& \leq & \sum_{j=1}^{\ell} \dist_{Q^c}(g_{j}(x_1),x_1)  \\
& \leq & \ell D_S.
\end{eqnarray*}
This shows that $D^{-1}n^{C/2}-Dn^{-C/2}\leq D_S\times \ell$ for  large values of $n$ (and $\ell$), and the conclusion follows.

\section{The Picard-Manin space and hyperbolic geometry}\label{par:Picard-Manin}

In this section, we recall the construction of the Picard-Manin space of a projective surface $X$
(see \cite{Cantat:SLC, Manin:cubic-forms} for details).

\subsection{Picard-Manin spaces}
Let $X$ be a smooth, irreducible, projective surface. We denote its  N\'eron-Severi group by $\Num(X)$;
when $\bfk=\C$, $\Num(X)$ can be identified to $H^{1,1}(X;\R)\cap H^2(X;\Z)$.
The  {\bf{intersection form}}
\begin{equation}
(C,D)\mapsto C\cdot D
\end{equation}
 is a non-degenerate quadratic form on $\Num(X)$ of signature $(1,\rho(X)-1)$. The Picard-Manin space $\ZZ(X)$ is 
 the limit $\lim_{\pi\colon X'\to X} \Num(X')$ obtained by looking at all birational morphisms $\pi\colon X'\to X$, where $X'$ is
 smooth and projective. 
By construction, $\Num(X)$ embeds naturally as a proper subspace of $\ZZ(X)$, and the intersection form 
is negative definite on the infinite dimensional space $\Num(X)^\perp$.

\begin{eg}\label{eg:bubble}
The group $\Pic(\P^2_\bfk)$ is generated by the class $\bfe_0$ of a line.
Blow-up one point $q_1$ of the plane, to get a morphism $\pi_1\colon X_1\to \P^2_\bfk$. Then, $\Pic(X_1)$ is a free abelian group of rank $2$, 
generated by the class $\bfe_1$ of the exceptional divisor $E_{q_1}$, and by the pull-back of $\bfe_0$ under $\pi_1$ (still denoted $\bfe_0$ in what follows). 
After $n$  blow-ups  one obtains
\begin{equation}\label{Eq:directsum}
\Pic(X_n)=\Num(X_n)=\Z \bfe_0 \oplus \Z \bfe_1 \oplus \ldots \oplus \Z \bfe_n
\end{equation}
where $\bfe_0$ (resp. $\bfe_i$) is the class of the total transform of a line (resp. of the exceptional divisor $E_{q_i}$) by the composite morphism $X_n\to \P^2_\bfk$
(resp. $X_n\to X_i$). 
The direct sum decomposition \eqref{Eq:directsum} is orthogonal with respect to the intersection form: 
\begin{equation}
\bfe_0\cdot\bfe_0=1,\quad\bfe_i\cdot\bfe_i=-1\;
\,\forall\,1\leq i\leq n,\quad\text{and}
\quad \bfe_i\cdot \bfe_j =0\;\,\forall\,0\leq i\neq j\leq n.
\end{equation}
Taking limits,   $\ZZ(\P^2_\bfk)$ splits as a direct sum 
$
\ZZ(\P^2_\bfk)= \Z \bfe_0 \oplus \bigoplus_q \Z \bfe_q
$
where $q$ runs over all possible points of the so-called
bubble space  $ {\mathcal{B}}(\P^2_\bfk)$ of $\P^2_\bfk$ (see \cite{Manin:cubic-forms, Dolgachev:CAG, Blanc-Cantat}).
\end{eg}

\subsection{The hyperbolic space $\HH_\infty(X)$}
Denote by $\ZZ(X,\R)$ and $\Num(X,\R)$ the tensor products   
$\ZZ(X)\otimes_\Z \R$ and  $\Num(X)\otimes_\Z \R$. Elements of $\ZZ(X,\R)$ are finite sums 
$ u_X + \sum_i a_i \bfe_i $
where $u_X$ is an element of $\Num(X,\R)$, each $\bfe_i$ is the class of an exceptional divisor, and the coefficients $a_i$ are real numbers. 
Allowing infinite sums with $\sum_i a_i^2<+\infty$, one gets a new space $\ZC(X)$, on which the intersection form extends
continuously \cite{Cantat:Annals, Boucksom-Favre-Jonsson}. 
Fix an ample class $\bfe_0$ in $\Num(X)\subset \ZZ(X)$. 
The subset of elements $u$ in $\ZC(X)$ such that  $u\cdot u=1$ is a hyperbolo\"{\i}d, and
  \begin{equation}
 \HH_\infty(X)=\{u \in \ZC(X)
 \; \vert \quad  u\cdot u=1
 \quad\text{and}\quad u\cdot\bfe_0>0\}
\end{equation}
is the sheet of that hyperboloid  containing ample classes of $\Num(X,\R)$.  
With the distance $\dist(\cdot, \cdot)$ defined by
\begin{equation}\label{eq:distance-intersection}
\cosh \dist(u,u')=u\cdot u',
\end{equation}
$ \HH_\infty(X)$ is isometric to the hyperbolic space $ \HH_\infty$ described in Section~\ref{par:Hyperb-space}. 

We denote by $\Isom(\ZC(X))$ the group of isometries of $\ZC(X)$ with respect to the intersection form, and by $\Isom( \HH_\infty(X))$ 
the subgroup that preserves $ \HH_\infty(X)$. As explained in \cite{Cantat:Annals, Cantat:SLC, Manin:cubic-forms}, the group 
$\Bir(X)$ acts by isometries on $\HH_\infty$. The homomorphism 
\begin{equation}
f\in \Bir(X) \mapsto f_\bullet\in \Isom( \HH_\infty(X))
\end{equation}
is injective. 
  
\subsection{Types and degree growth}\label{par:Types-Growth}
Since $\Bir(X)$ acts faithfully on $ \HH_\infty(X)$, there are three types of birational transformations: {\bf{Elliptic}}, {\bf{parabolic}}, and {\bf{loxodromic}}, according to the type of the associated isometry of $ \HH_\infty(X)$. We now describe how each type can be characterized in algebro-geometric terms.

\subsubsection{Degrees, distances, translation lengths and loxodromic elements} Let ${\bfh}\in \Num(X, \R)$ be an ample class with
self-intersection $1$.  The degree of $f$ with respect to the polarization $\bfh$ is
$\deg_{\bfh}(f)= f_\bullet({\bfh})\cdot \bfh=\cosh(\dist(\bfh, f_\bullet\bfh)).$
Consider for instance an element $f$ of $\Bir(\P^2_\bfk)$, with the polarization $\bfh=\bfe_0$ given from the class of a line; 
then the image of a general line by $f$ is a curve of degree $\deg_\bfh(f)$ which goes through
the base points $q_i$ of $f^{-1}$ with certain multiplicities $a_i$, and 
\begin{equation}
f_\bullet\bfe_0=\deg_\bfh(f) \bfe_0-\sum_i a_i\bfe_i
\end{equation}
where $\bfe_i$ is the class corresponding to the exceptional divisor that one gets when blowing up the point $q_i$. 

 If the translation length $L(f_\bullet)$ is positive, we know that the distance $\dist(f_\bullet^n(x), x)$ grows like $nL(f_\bullet)$ for every $x\in  \HH_\infty(X)$ 
(see Section~\ref{par:Hyperb-space}). We get: {\sl{
the   logarithm $\log(\lambda_1(f))$ of the dynamical degree of $f$ is
the translation length $L(f_\bullet)$ of the isometry $f_\bullet$.}}
In particular, $f$ is loxodromic if and only if $\lambda_1(f)>1$.

\subsubsection{Classification} Elliptic and parabolic transformations are also classified in terms of degree growth.
Say that a sequence of real numbers $(d_n)_{n\geq 0}$
grows linearly (resp. quadratically) if $n/c\leq d_n \leq cn$ (resp. $n^2/c\leq d_n \leq cn^2$) for some $c>0$. 

\begin{thm}[Gizatullin, Cantat, Diller and Favre, see \cite{Gizatullin:1980, Cantat:Acta, Cantat:Trans, Diller-Favre:2001}]\label{thm:Elements} Let $X$ be a projective surface, defined over an algebraically closed field $\bfk$, and $\bfh$ be a polarization of $X$. Let $f$ be a birational transformation of $X$.
\begin{enumerate}
\item $f$ is elliptic if and only if the sequence $\deg_\bfh(f^n)$ is bounded. In this case, there exists a birational map $\phi\colon Y\dasharrow X$
and an integer $k\geq 1$ such that $\phi^{-1} \circ f\circ \phi $ is an automorphism of $Y$ and $\phi^{-1} \circ f^k\circ \phi $ is in the connected component
of the identity of the group $\Aut(Y)$.
\item $f$ is parabolic if and only if the sequence $\deg_\bfh(f^n)$ grows linearly or quadratically with $n$. If $f$ is parabolic,
there exists a birational map $\psi\colon Y\dasharrow X$ and a  fibration $\pi\colon Y\to B$ onto a curve $B$ such that 
$\psi^{-1}\circ f \circ \psi$ permutes the fibers of $\pi$. The fibration is rational if the growth is linear, and  elliptic (or quasi-elliptic if $char(\bfk)\in\{2,3\}$) 
if the growth is quadratic.
\item $f$ is loxodromic if and only if $\deg_\bfh(f^n)$ grows exponentially fast with $n$: There is a constant $b_\bfh(f)>0$ such that $\deg_\bfh(f^n)=b_\bfh(f)\lambda(f)^n+O(1)$.
\end{enumerate}
\end{thm}

\subsection{Elliptic elements of $\Cr_2(\bfk)$}

Every elliptic, infinite order element of $\Bir(\P^2_\bfk)$ is conjugate to an automorphism $f\in \PGL_3(\bfk)$ 
when $\bfk$ is algebraically closed (see~\cite{Blanc-Deserti:2015}). Thus, Theorem~\ref{thm:dbl-expo} stipulates that elliptic elements 
of infinite order are 
\begin{itemize}

\item exactly doubly exponentially distorted if they are conjugate to a virtually unipotent element of $\PGL_3(\bfk)$;

\item exactly exponentially distorted otherwise. 
\end{itemize}

\subsection{Loxodromic elements of $\Cr_2(\bfk)$}
Loxodromic elements have an exponential degree growth; by Proposition~\ref{degd}, they are not distorted. This result 
applies to all loxodromic elements $f\in \Bir(X)$, for all projective surfaces.

\subsection{Parabolic elements of $\Cr_2(\bfk)$}

According to Theorem~\ref{thm:Elements}, there are two types of parabolic elements, depending on the growth of the sequence $\deg(f^n)$: 
Jonqui\`eres and Halphen twists. Here, we collect extra informations on these transformations, and study their distortion properties in 
Sections~\ref{par:Par-Horoball} and~\ref{par:Jonq-Non-Disto}. 

 \subsubsection{Jonqui\`eres twists}\label{par:Jonquieres-twists}
  Let $f$ be an element of $\Cr_2(\bfk)$ for which the sequence $\deg(f^n)$
 grows linearly with $n$. Then, $f$ is called a {\bf{Jonqui\`eres twist}}. Examples are given by  the transformations 
 $f(X,Y)=(X,Q(X)Y)$ with $Q\in \bfk(X)$ of degree $\geq 1$. The following properties follow from \cite{Blanc-Cantat, Blanc-Deserti:2015, Diller-Favre:2001}. \smallskip
 
\noindent{\bf{Normal form}}.-- There is a birational map $\varphi\colon \P^1_\bfk\times\P^1_\bfk\dasharrow \P^2_\bfk$ 
that conjugates $f$ to an element $g$ 
 of $\Bir(\P^1_\bfk\times\P^1_\bfk)$ which preserves the projection $\pi\colon \P^1_\bfk\times\P^1_\bfk\to \P^1_\bfk$
 onto the first factor. More precisely, there is an automorphism $A$ of $\P^1_\bfk$ such that  $\pi\circ g =A\circ \pi$. If $x$ and $y$ are 
 affine coordinates on each of the factors, then 
 \begin{equation}
 g(x,y)=(A(x),B(x)(y))
 \end{equation}
 where $(A,B)$ is an element of the semi-direct product $\PGL_2(\bfk)\ltimes \PGL_2(\bfk(x))$. 
 Alternatively, $f$ is conjugate to an element $g'$ of $\Cr_2(\bfk)$ that preserves the pencil of lines through the point $[0:0:1]$.\smallskip

\noindent{\bf{Action on $\HH_\infty(\P^2_\bfk)$}}.--
Assume now that $g'$ preserves the pencil of lines through the point $q_1:=[0:0:1]$. Let $\bfe_1\in \ZZ(\P^2_\bfk;\R)$ be the 
class of the exceptional divisor $E_1$ that one gets by blowing-up $q_1$. Then $g'_\bullet$ preserves the 
isotropic vector $\bfe_0-\bfe_1$ (corresponding to the class of the linear system of lines through $q_1$), and
 the unique fixed point of $g'_\bullet$ on $\partial \HH_\infty(\P^2_\bfk)$ is determined by $\bfe_0-\bfe_1$.
Let $d$ denote the degree $\deg_{\bfe_0}(g')$. Let $q_i$ denote the base points of $(g')^{-1}$ (including 
infinitely near base points) and $\bfe(q_i)$ be the corresponding classes of exceptional divisors. 
From \cite{Blanc-Cantat, Alberich}, one knows that there are $2d-1$ base points (including $q_1$), and that
\begin{eqnarray}
g'_\bullet\bfe_0 & = & d\bfe_0-(d-1)\bfe(q_1)-\sum_{i=2}^{2d-1} \bfe(q_i) \\
g'_\bullet\bfe(q_1) & = & (d-1)\bfe_0-(d-2)\bfe(q_1)-\sum_{i=2}^{2d-1} \bfe(q_i).
\end{eqnarray}
 
\noindent{\bf{Degree growth}}.--  
The sequence $\frac{1}{n}\deg_{\bfe_0}(f^n)$ converges toward a number $\alpha(f)$. The 
 set $\{\alpha(hfh^{-1}); h\in \Cr_2(\bfk)\}$ admits a minimum; this minimum is of the form $\frac{1}{2}\mu(f)$
 for some integer $\mu(f)>0$, and there is an integer $a\geq 1$ such that $\alpha(f)=  \frac{1}{2}\mu(f) a^2$.
Blanc and D\'eserti prove also that $a=1$ precisely when $f$ preserves a pencil of lines in $\P^2_\bfk$ (thus, the conjugate
$g'$ of $f$ satisfies $\alpha(g')= \frac{1}{2}\mu(f) $).  
Moreover, when $f$ preserves such a pencil, one knows from \cite{Blanc-Cantat}, Lemma~5.7, that
$\deg_{\bfe_0}(f^n)$ is a subadditive sequence. Thus, $\frac{1}{n}\deg_{\bfe_0}(f^n)\geq \mu/2$, and 
$\mu/2$ is the infimum of $\frac{1}{n}\deg_{\bfe_0}(f^n)$. In Section~\ref{par:Jonq-Non-Disto}, 
we shall describe how Blanc and D\'eserti interpret  $\mu(f)$ 
as an asymptotic number of base points.  

 \subsubsection{Halphen twists}\label{par:Halphen-twists}
  Let $f$ be an element of $\Cr_2(\bfk)$ for which the sequence $\deg(f^n)$
 grows quadratically with $n$. Then, $f$ is called a {\bf{Halphen twist}}. The following properties follow from \cite{Blanc-Cantat, Cantat-Dolgachev, Cantat-Favre}. \smallskip

\noindent{\bf{Normal form.--}} There is a rational surface $X$, together with a birational map $\varphi\colon X\dasharrow \P^2_\bfk$
and a genus~$1$ fibration $\pi\colon X\to \P^1_\bfk$ such that $g=\varphi^{-1}\circ f \circ \varphi$ is a regular automorphism
of $X$ that preserves the fibration $\pi$. More precisely, there is an element $A$ in $\Aut(\P^1_\bfk)$ of finite order such that 
$\pi \circ g = A\circ \pi.$ Changing $g$ into $g^k$ where $k$ is the order of $g$, we may assume that the action on the base of $\pi$ is trivial; then, 
$g$ acts by translations along the fibers of $\pi$. 

There is a classification of genus~$1$ pencils of the plane up to birational conjugacy, which dates back to Halphen 
(see \cite{Dolgachev:1966, Grivaux:2016}): 
a Halphen pencil of index~$l$ is a pencil of curves of degree
$2l$ with $9$ base-points of multiplicity $l$. Every Halphen twist  $f$ preserves such a pencil; on $X$, the pencil corresponds 
to the genus~$1$ fibration which is $g$-invariant. \smallskip

\noindent{\bf{Action on $\HH_\infty(\P^2_\bfk)$ and degree growth.--}}
Let $\bfc$ be the class of the fibers of $\pi$ in $\Num(X)$ (resp. in $\ZZ(X)=\ZZ(\P^2_\bfk)$). 
This class is $g$-invariant (resp. $f_\bullet$-invariant) and isotropic. Thus $\bfc\in \ZZ(\P^2_\bfk)$ determines
the unique fixed point of the parabolic isometry $f_\bullet$ on $\partial \HH_\infty(\P^2_\bfk)$. 

After conjugacy, we may assume that the genus~$1$ fibration $\pi$ comes from a Halphen pencil 
of the plane of index $l$ with nine base points $q_1$, $\ldots$, $q_9$. This linear system corresponds
to the class $\bfc$ such that
\begin{equation}
\frac{1}{l}\bfc=3\bfe_0- \sum_{j=1}^9 \bfe(q_j).
\end{equation}  
Thus, after conjugacy, we may assume that the Halphen twist $g$ fixes such a class. 
Under this hypothesis, Lemma~5.10 of \cite{Blanc-Cantat} provides the following 
inequality 
\begin{equation}
\sqrt{\deg_{\bfe_0}(g^{n+m})}\leq \sqrt{\deg_{\bfe_0}(g^{n})} + \sqrt{\deg_{\bfe_0}(g^{m})}
\end{equation}
for all integers $n,m \geq 0$. In particular, the number 
\begin{equation}
\tau(g)=\inf_{n>0}\frac{1}{n} \sqrt{\deg_{\bfe_0}(g^n)}=\lim_{n\to + \infty} \frac{1}{n} \sqrt{\deg_{\bfe_0}(g^n)}
\end{equation}
is a well defined positive real number, and 
$
\deg_{\bfe_0}(g^n)\geq \tau(g) n^2
$
for all $n\geq 1$. 
Blanc and D\'eserti prove that the minimum $\kappa(g)=\min \tau(hgh^{-1})^2$ for ${h\in \Cr_2(\bfk)}$ is a positive rational number and that
$\lim_{n\to +\infty}\frac{1}{n^2}\deg_{\bfe_0}(g^n)=\frac{\kappa(g)}{9}a^2
$ for some integer $a\geq 3$. 

\section{Parabolic elements of $\Cr_2(\bfk)$ and their invariant horoballs}\label{par:Par-Horoball}

For simplicity, the hyperbolic space $\HH_\infty(\P^2_\bfk)$ will be denoted by $\HH_\infty$. In this section, 
we prove Theorem~\ref{thc}, which states that sufficiently small horoballs invariant by Jonqui\`eres or Halphen twists
are pairwise disjoint. Combined with Theorem~\ref{thb}, this result implies that Halphen twists are not distorted.

\subsection{Small horoballs associated to Halphen and Jonqui\`eres twists} 
\subsubsection{Fixed points of Jonqui\`eres and Halphen twists} 
Let $f$ be an element of $\Cr_2(\bfk)$ acting as a parabolic isometry on the hyperbolic space 
$\HH_\infty$. Then, $f$ fixes a unique point $\xi$ on the boundary $\partial \HH_\infty$. Up 
to conjugacy, there are two possibilities: 
\begin{itemize}
\item $f$ is a Jonqui\`eres twist, and $f$ preserves the pencil of lines through a point $q_1$ of $\P^2_\bfk$. Then, setting $\bfe_1=\bfe(q_1)$, the
boundary point $\xi$ is represented by the ray $\R^+w$, where
\begin{equation}\label{Eq:wJ}
w_J=\bfe_0-\bfe_1.
\end{equation}

\item $f$ is a Halphen twist. Then, up to conjugacy, $\xi$ is $\R^+w$ with 
\begin{equation}\label{Eq:wH}
w_H=3\bfe_0-\bfe_1-\bfe_2-\bfe_3-\bfe_4-\bfe_5-\bfe_6-\bfe_7-\bfe_8-\bfe_9,
\end{equation}
where the $\bfe_i$ are the classes given by the blow-up of the base-points of a Halphen 
pencil. 
\end{itemize}

\subsubsection{Disjonction of horoballs}
 If $w$ is an element of the Picard-Manin space with $w^2=0$ and $w\cdot \bfe_0>0$, 
the ray $\R^+w$ determines a boundary point of $\HH_\infty$. Let $\epsilon$ be a positive
real number. The horoball $H_w(\epsilon)$ is defined in Section~\ref{par:horoballs}; its elements are characterized by the following three constraints: 
\begin{equation}
v^2=1, \quad v\cdot \bfe_0 > 0, \quad 0 < v\cdot w < \epsilon.
\end{equation}
When $f$ is a Jonqui\`eres or Halphen twist then, after conjugacy, $f_\bullet$ preserves the horoballs centered $H_{w_J}(\epsilon)$ or $H_{w_H}(\epsilon)$.
Define 
\begin{equation}
\epsilon_J=\frac{\sqrt{3}-1}{\sqrt{2}}\simeq 0.5176 \quad \textrm{ and }\quad\epsilon_H:=\frac{1}{3\sqrt{2}}\simeq 0.2357
\end{equation}
 
\begin{thm}\label{thc}\label{thm:disjoint-horoballs}
Let $w_J$ be the class $\bfe_0-\bfe_1\in \HH_\infty(\P^2_\bfk)$ determined by the pencil of lines through a point $q_1$. If $0< \epsilon < \epsilon_J$, the horoballs 
$h(H_{w_J}(\epsilon))$, for $h\in \Cr_2(\bfk)$, are pairwise disjoint; more precisely, given $h$ in 
$\Cr_2(\bfk)$, 
\[
 \textrm{ either }\quad h(H_{w_J}(\epsilon))=H_{w_J}(\epsilon) \quad\mathrm{ or }\quad h(H_{w_J}(\epsilon))\cap H_{w_J}(\epsilon)=\emptyset.
\]

Let $w_H$ be the class $3\bfe_0-\bfe_1-\bfe_2-\bfe_3-\bfe_4-\bfe_5-\bfe_6-\bfe_7-\bfe_8-\bfe_9$ determined by a Halphen pencil. 
If $0< \epsilon \leq \epsilon_H$, the horoballs 
$h(H_{w_H}(\epsilon))$, $h\in \Cr_2(\bfk)$, are pairwise disjoint; more precisely, given $h$ in 
$\Cr_2(\bfk)$, 
\[
 \textrm{ either }\quad h(H_{w_H}(\epsilon))=H_{w_H}(\epsilon)\quad\textrm{ or }\quad h(H_{w_H}(\epsilon))\cap H_{w_H}(\epsilon)=\emptyset.
\]
\end{thm}

\subsection{Proof of the first assertion}

\subsubsection{}
For simplicity, we write $w$ instead of $w_J$.
Let $h$ be an element of $\Cr_2(\bfk)$. If $h_\bullet$ fixes the line $\R_+ w$, then it fixes $w$ and its dynamical degree
is equal to $1$; thus, $h$ fixes the horoballs $H_{w}(\epsilon)$. We may therefore assume that $h_\bullet$ 
does not fix $w$. Write 
\begin{equation}\label{eq:h-e0-e1}
h_\bullet(w)=h_\bullet(\bfe_0-\bfe_1)= m \bfe_0- \sum_i r_i \bfe_i
\end{equation}
for some multiplicities $r_i$ in $\Z^+$. Since $w^2=0$, we get
\begin{equation}
m^2=\sum_i r_i^2.
\end{equation}
For later purpose, we shall write $r_1=m-s_1$ for some integer $s_1\geq 0$. Then, 
\begin{equation}\label{eq:s1-and-rj2}
s_1^2+\sum_{j\geq 2} r_j^2=2ms_1.
\end{equation}
\begin{rem}
We have $h_\bullet(w)=w$ if and only if $m=1$ and $r_1=1$, if and only if $s_1=0$. Indeed, 
if $s_1=0$, then the last equation implies that all $r_j$ vanish for $j\geq 2$. Hence, 
$h_\bullet(w)=mw$ for some $m\geq 1$, $h$ is parabolic, and $m$ must be equal to the dynamical degree of $h$, 
so that $m=1$. 
\end{rem}

\subsubsection{}
Assume that $h_\bullet(H_{w}(\epsilon))$ intersects $H_{w}(\epsilon)$. Then, there exists a point $u$ in the
intersection. Write 
\begin{equation}
 u=\alpha_0e_0-\sum_i \alpha_i e_i.
\end{equation}
By definition of $H_w(\epsilon)$, we have 
$0  <   w\cdot u  <  \epsilon $ and $0  <   h_\bullet(w)\cdot u  <  \epsilon$, i.e. 
\begin{equation}\label{eq:ineq-h-horoball}
0  <   \alpha_0-\alpha_1    <  \epsilon 
\quad \textrm{and} \quad
0  <   m\alpha_0 -\sum_i r_i \alpha_i  <  \epsilon
\end{equation}
We shall write $\alpha_1=\alpha_0-\tau $
with $0< \tau < \epsilon$. Since $u\cdot e_0>0$ we know that $\alpha_0>0$, and since $u^2=1$ we 
have 
\begin{equation}\label{eq:sum-alphai2}
 \sum_i \alpha_i^2 = \alpha_0^2-1,
\end{equation}
and therefore 
\begin{equation}\label{eq:tau-and-alphaj2}
\tau^2+\sum_{j\geq 2} \alpha_j^2=2\alpha_0\tau - 1. 
\end{equation}

\subsubsection{}

In a first step, we prove a lower estimate for $\alpha_0$. By Equation~\eqref{eq:ineq-h-horoball},
\begin{equation}
m\alpha_0 < \epsilon+ \sum_i \alpha_ir_i.
\end{equation}
Apply Cauchy-Schwartz inequality and use Equations~\eqref{eq:h-e0-e1} and~\eqref{eq:sum-alphai2} to obtain 
\begin{equation}
m\alpha_0 < \epsilon+  (\sum_i\alpha_i^2)^{1/2}(\sum_ir_i^2)^{1/2} = \epsilon+ (\alpha_0^2-1)^{1/2}(m^2)^{1/2}.
\end{equation}
This gives 
\begin{equation}
m\alpha_0 (1-(1-1/\alpha_0^2)^{1/2})< \epsilon.
\end{equation}
Then, remark that $(1-t)^{1/2}\leq 1-t/2$, to deduce 
$1-(1-1/\alpha_0^2)^{1/2}\geq \frac{1}{2\alpha_0^2}$, and
inject this relation in the previous inequality to get 
\begin{equation}
\frac{m}{2\epsilon}< \alpha_0.
\end{equation}

\subsubsection{}

Isolate  $r_1\alpha_1$ in Equation~\eqref{eq:ineq-h-horoball}, i.e. write
$m\alpha_0 - r_1\alpha_1 - \sum_{j\geq 2} \alpha_j r_j< \epsilon$,
to obtain
\begin{equation}
s_1\alpha_0 + m\tau< \epsilon +  s_1\tau + \sum_{j\geq 2} \alpha_j r_j .
\end{equation}
Then, remark that $m\tau \geq 0$, and apply  Cauchy-Schwartz estimate to the vectors $(s_1, (r_j)_{j\geq 2})$
and $(\tau, (\alpha_j)_{j\geq 2})$; from Equations~\eqref{eq:tau-and-alphaj2} and~\eqref{eq:s1-and-rj2} we get
\begin{equation}
s_1\alpha_0 < \epsilon + (2\alpha_0\tau -1)^{1/2} (2m s_1)^{1/2} 
\end{equation}
\begin{equation}
< \epsilon + 2 (\alpha_0\epsilon)^{1/2} (ms_1)^{1/2}
\end{equation}
because $0< \tau < \epsilon$. This gives 
\[
\left(\frac{s_1}{m}\alpha_0\right)^{1/2}< \frac{\epsilon}{(ms_1\alpha_0)^{1/2}} + 2(\epsilon)^{1/2} 
\]
and the inequality $\alpha_0> m/(2\epsilon)$ gives 
\[
\left(\frac{s_1}{2\epsilon}\right)^{1/2} < \frac{\epsilon^{3/2}}{(m^2s_1/2)^{1/2}} + 2(\epsilon)^{1/2}.
\]
Since $s_1\geq 1$ and $m\geq 1$ we get $(\sqrt{2})^{-1} < \sqrt{2} \epsilon^{2} + 2\epsilon $, in contradiction with $\epsilon < \epsilon_J$.

\subsection{Proof of the second assertion}

The proof follows  the same lines. 

\subsubsection{}
For simplicity, we write $w$ instead of $w_H$.
Let $h$ be an element of $\Cr_2(\bfk)$. If $h_\bullet$ the line $\R w$, it fixes also the class $w$, and its dynamical degree
is equal to $1$; thus, $h_\bullet$ fixes the horoballs $H_{w}(\epsilon)$. Thus, we may assume that $h_\bullet$ 
does not fix $w$. Write 
\begin{equation}
h_\bullet(w)= m e_0- \sum_i r_i e_i
\end{equation}
for some $r_i$ in $\Z^+$. Since $w^2=0$, we get 
\begin{equation}
m^2=\sum_i r_i^2.
\end{equation}
For later purpose, we shall write $r_i=(m/3)-s_i$ for each index $1\leq i \leq 9$. Then
\begin{equation}
\sum_{i=1}^9 s_i^2+\sum_{j\geq 10} r_j^2=(2/3)m S.
\end{equation}
with 
\begin{equation}
S:=\sum_{i=1}^9s_i.
\end{equation}
\begin{rem}
We have $h_\bullet(w)=w$ if and only if $m=3$ and $r_i=1$ for $1\leq i \leq 9$. This is equivalent 
to $S=0$. Indeed, if $S=0$, then the last inequality implies that all multiplicities $r_j$ vanish
for $j\geq 10$, and all $s_i$ vanish for $1\leq i\leq 9$. Thus, $h_\bullet(w)=mw$, $m$ must be equal to the dynamical degree of $h$, 
and $m=1$. 
\end{rem}

\subsubsection{}
Assume that $h_\bullet(H_{w}(\epsilon))$ intersects $H_{w}(\epsilon)$. Then, there exists a point $u$ in the
intersection. Write 
$u=\alpha_0e_0-\sum_i \alpha_i e_i$.
By definition, we have $0  <   w\cdot u  <  \epsilon$ and $0  <   h_\bullet(w)\cdot u  <  \epsilon$, i.e. 
\begin{equation}\label{eq:Halp-alphai2}
0  <   3\alpha_0-\sum_{i=1}^9 \alpha_i    <  \epsilon 
\quad\textrm{ and }\quad
0  <   m\alpha_0 -\sum_i r_i \alpha_i  <  \epsilon
\end{equation}
We shall write $\alpha_i=(1/3)\alpha_0-\tau_i$ for $1\leq i \leq 9$, and 
$T=\sum_{i=1}^9 \tau_i$. Then, 
\begin{equation}
0< T < \epsilon.
\end{equation}
Since $u\cdot e_0>0$ we know that $\alpha_0>0$, and since $u^2=1$ we 
have 
\begin{equation}
 \sum_i \alpha_i^2 = \alpha_0^2-1.
\end{equation}
Thus,
\begin{equation}
\sum_{i=1}^9\tau_i^2+\sum_{j\geq 10} \alpha_j^2=(2/3)\alpha_0T - 1. 
\end{equation}

\subsubsection{}

The following lower estimate is obtained as in the case $w=w_J$:
\begin{equation}
\frac{m}{2\epsilon}< \alpha_0.
\end{equation}

\subsubsection{}

Now, isolate the terms $r_i\alpha_i$, for $i$ between $1$ and $9$, in Equation~\eqref{eq:Halp-alphai2}:
\begin{equation}
m\alpha_0 - \sum_{i=1}^9 r_i\alpha_i - \sum_{j\geq 10} \alpha_j r_j< \epsilon
\end{equation}
We obtain
\begin{equation}
(m-(1/3)\sum_{i=1}^9 r_i)\alpha_0 + \sum_i r_i \tau_i  <  \epsilon  + \sum_{j\geq 10} \alpha_j r_j 
\end{equation}
i.e. 
\begin{equation}
(1/3)S \alpha_0 + (1/3)mT < \epsilon + \sum_{i=1}^9 s_i \tau_i + \sum_{j\geq 10} \alpha_j r_j
\end{equation}
Apply again, the fact that $mT \geq 0$ and Cauchy-Schwartz estimate: 
\begin{equation}
(1/3)S \alpha_0 -\epsilon  <  ((2/3)\alpha_0 T -1)^{1/2} ((2/3) m S)^{1/2} <  (2/3) (\alpha_0\epsilon)^{1/2} (mS )^{1/2}
\end{equation}
because $0< T < \epsilon$. This gives 
\begin{equation}
\frac{1}{3}\left(\frac{S}{m}\alpha_0\right)^{1/2}< \frac{\epsilon}{(mS\alpha_0)^{1/2}} + \frac{2}{3}(\epsilon)^{1/2} 
\end{equation}
and the inequality $\alpha_0> m/(2\epsilon)$ implies 
\begin{equation}
\frac{1}{3}\left(\frac{S}{2\epsilon}\right)^{1/2} < \frac{\epsilon^{3/2}}{(m^2 S /2)^{1/2}} + \frac{2}{3}(\epsilon)^{1/2} 
\end{equation}
Since $S\geq 1$ and $m\geq 1$, we get $(3\sqrt{2})^{-1}< \sqrt{2} \epsilon^2+(2/3)\epsilon$, in contradiction with $\epsilon < \epsilon_H$.

\subsection{Consequence: Halphen twists are not distorted}\label{par:halphen-undistorted}

Let $h\in \Cr_2(\bfk)$ be a Halphen twist. After conjugacy, we may assume
that $h_\bullet$ preserves the class $w_H$ associated to some Halphen pencil. 
We know from Section~\ref{par:Halphen-twists} that the degree growth of $h$ is quadratic, 
with 
\begin{equation}
\deg_{\bfe_0}(h^n)\geq (\tau(h)n)^2.
\end{equation}
Since $\deg_{\bfe_0}(h^n)$ 
is equal to $\cosh(\dist(h_\bullet\bfe_0,\bfe_0))$, we obtain the lower bound
\begin{equation}
\log \dist(h_\bullet\bfe_0,\bfe_0)) \geq 2 \log(n) - 2\log(\tau(h)).
\end{equation}
Set $\HH_m=\HH_\infty(\P^2_\bfk)$, $f=g_\bullet$, $G=\Cr_2(\bfk)$, $B=H_{w_H}(\epsilon_H/2)$, and $C=2$. By Theorem~\ref{thc} if $g$ is an element of $G$ then 
$g(B)=B$ or $g(B)\cap B=\emptyset$. Thus, we may apply Theorem~\ref{thb} to $f=h_\bullet$
and we get the desired result: $h$ is undistorted in $\Cr_2(\bfk)$.

\subsection{Non-rational surfaces}

The previous paragraph makes use of the explicit description of Halphen pencils in $\P^2_\bfk$. Here, we consider 
a smooth projective surface $X$, over the algebraically closed field $\bfk$, and assume that 
\begin{itemize}
\item $X$ is not rational;
\item $f$ is a birational transformation of $X$ with $\deg(f^n)\simeq n^2$ (we shall say that $f$ is a Halphen twist of $X$). 
\end{itemize}
Then, from Theorem~\ref{thm:Elements}, we know that $f$ preserves a unique pencil of genus $1$. 
\begin{lem}
The Kodaira dimension of $X$ is equal to $0$ or $1$. The surface $X$ has a unique minimal model $X_0$, and
$\Bir(X_0)=\Aut(X_0)$. 
\end{lem} 
\begin{proof}
A Halphen twist has infinite order, thus $\Bir(X)$ is infinite, and the Kodaira dimension of $X$ is $<2$. 
If it is equal to $-\infty$, then $X$ is a ruled surface, and since $X$ is not rational, the ruling is unique
and $\Bir(X)$-invariant. Thus, $f$ must preserve two pencils. These two rational fibrations determine two $f_\bullet$-invariant isotropic classes in $\ZZ(X)$, 
in contradiction with the fact that $f_\bullet$ is parabolic. This proves the first assertion. 
The second one is a well-known consequence of the first. 
\end{proof}

We can therefore conjugate $f$ to an automorphism $f_0$ of $X_0$, and assume that $\Bir(X_0)=\Aut(X_0)$. 
Thus, the distortion of $f$ in $\Bir(X)$ is now equivalent to the distortion of $f_0$ in $\Aut(X_0)$. Instead of 
looking at the infinite dimensional vector space $\ZZ(X)$, we can look at the action of $\Aut(X_0)$ on the
N\'eron-Severi group $\Num(X_0)$.

 Identify $\Num(X_0)$ to $\Z^r$, where $r$ is the Picard number of $X$, and 
denote by $q_0$ the intersection form on $\Num(X_0)$. 
Then, the image of $\Aut(X_0)$ in $\GL(\Num(X))$ is a subgroup of the orthogonal group $O^+(q_0;\Z)$ preserving the 
 hyperbolic space $\HH_r\subset \Num(X_0;\R)$ defined by $q_0$. The quotient $V=\HH_r/O^+(q_0;\Z)$ is a hyperbolic orbifold, and 
 the  fixed point $\xi$ of $f_0$ in $\partial\HH_r$ gives a cusp of $V$. A sufficiently small horoball $B$ centered at $\xi$ determines
 a neighborhoods of this cusp (see~\cite{Ratcliffe}). Thus, if $g$ is an element of $O^+(q_0;\Z)$, then $g(B)=B$ or 
 $g(B)\cap B=\emptyset$, as in Theorem~\ref{thm:disjoint-horoballs}. From Theorem~\ref{thb}, we deduce that $f$ is undistorted. 
 We have proved:
 
 \begin{thm}\label{thm:Halphen-Undistorted}
 Let $\bfk$ be an algebraically closed field. Let $X$ be a smooth projective surface, defined over $\bfk$. 
 If $f\in \Bir(X)$ is a Halphen twist (i.e. $\deg(f^n)\simeq n^2$), then $f$ is not distorted in $\Bir(X)$.
 \end{thm}

\section{Jonqui\`eres twists are undistorted}\label{par:Jonq-Non-Disto}

The argument presented in Section~\ref{par:halphen-undistorted} to show that Halphen twists are undistorted 
is not sufficient for Jonqui\`eres twists; it only gives a quadratic upper bound on the distortion function.
As we shall see, the following result follows from \cite{Blanc-Deserti:2015}. 

\begin{thm}\label{thm:Jonquieres-Undistorted}
Let $\bfk$ be an algebraically closed field, and let $X_\bfk$ be a projective surface.
Let $f$ be an element of $\Bir(X)$. If $f$ is a  Jonqui\`eres twist (i.e. if the sequence $\deg(f^n)$ 
grows linearly) then $f$ is 
not distorted in $\Bir(X)$. 
\end{thm}

\subsection{In the Cremona group}
 
We first describe the proof when $X$ is the projective plane.  Denote by 
$\bp\colon \Cr_2(\bfk) \to \Z_+$
the function {\bf{number of base-points}}:  $\bp(f)$ is the
number of base-points of the homaloidal net of $f$, i.e. of the linear system
of curves obtained by pulling-back the system of lines in $\P^2$. 
Indeterminacy points are examples of base-points, but the base-point set
may also include infinitely near points. The number of base-points is also 
the number of blow-ups needed to construct a minimal resolution of the
indeterminacies of $f$. If $f_\bullet$ denotes the action of $f$ on the
Picard-Manin space, and $e_0$ is the class of a line, 
then 
\begin{equation}
(f^{-1})_\bullet e_0=de_0-\sum_i m_i e(p_i)
\end{equation}
where $d$ is the degree of $f$ and $m_i$ is the multiplicity of the homaloidal
system $f^*{\mathcal{O}}(1)$ at the base-point $p_i$; thus, $\bp(f)$ is just the number
of classes   for which the multiplicity $m_i$ is positive. 
The number of base-points is non-negative, is subadditive, and is symmetric (see  \cite{Blanc-Deserti:2015}):
$\bp(f\circ g)\leq \bp(f)+\bp(g)$ and  $\bp(f)=\bp(f^{-1})$. 
As a consequence, the limit 
\begin{equation}
\alpha(f)=\lim_{n\to +\infty} \frac{1}{n}\bp(f^n)
\end{equation}
exists and is non-negative. It is symmetric, i.e. $\alpha(f^{-1})=\alpha(f)$, invariant under 
conjugacy, and it vanishes if $f$ is distorted, because if $f$ is distorted its stable length
vanishes (Lemma~\ref{lem:stable-length}) and this implies $\alpha(f)=0$ by the subadditivity of $\bp$. 

Blanc and D\'eserti prove that $\alpha(f)$ is a non-negative integer, 
and that it vanishes if and only if $f$ is conjugate to an automorphism by a birational 
map $\pi\colon X\dasharrow \P^2$. In particular, $\bp(f)>0$ for Jonqui\`eres twists
because  they are not conjugate to automorphisms; but the result of 
Blanc and D\'eserti is even more precise: {\sl{if $f$ is a Jonqui\`eres twist, $\alpha(f)$
co\"{\i}ncides with the integer $\mu(f)$  which was defined in Section~\ref{par:Jonquieres-twists}.}}
Theorem~\ref{thm:Jonquieres-Undistorted} follows from those results. 

\subsection{In $\Bir(X)$}

The definition of $\bp(f)$ extends to birational transformations of arbitrary smooth surfaces; again, 
its stable version $\alpha(f)$ is invariant under conjugacy, vanishes when $f$ is distorted, and may 
be interpreted as the number of terms in the decomposition $f^n\bfe_0={\mathbf{u}}_X+\sum_i a_i\bfe_i$
in the Picard-Manin space $\ZZ(X)=\Num(X) \oplus_i \Z\bfe_i$ (see Section~\ref{par:Picard-Manin}), where $\bfe_0$ is any ample class in $\Num(X)$. 
(The proofs of Blanc and D\'eserti extend directly to this general situation.) 

If $f\in \Bir(X)$ is a parabolic element with $\deg(f^n)\simeq n$, and if $X$ is not a rational 
surface, one can do a birational conjugacy to assume that $X$ is the product $C\times \P^1_\bfk$ of a curve of genus $g(C)\geq 1$ 
with the projective line. Then, $\Bir(X)$ preserves the projection $\pi\colon X\to C$, acting by automorphisms on the base. 

The N\'eron-Severi group of $X$ has rank $2$, and is generated by the class $\bfv$ of a vertical line $\{x_0\}\times\P^1_\bfk$ 
and by the class $\bfh$ of a horizontal section $C\times\{y_0\}$. The canonical class $\bfk_X$ is $2(g-1)\bfv-2\bfh$, where
$g$ is the genus of the curve $C$. Blowing-up $X$, the canonical class of the surfaces $X'\to X$ determines a limit 
\begin{equation}
{\tilde{\bfk}}=2(g-1)\bfv-2\bfh+\sum_i \bfe_i 
\end{equation}
where the $\bfe_i$ are the classes of all exceptional divisors, as in Section~\ref{par:Picard-Manin}. This limit is not an element of the Picard-Manin space
$\ZC(X)$, but it determines a linear form on the $\Z$-module $\ZZ(X)$; this form is invariant under the action of $\Bir(X)$ on $\ZZ(X)$. 

As an ample class, take $\bfe_0=\sqrt{2}^{-1} (\bfv+\bfh)$. This is an element of $\HH_\infty(X)$. If $f$ is an 
element of $\Bir(X)$, it preserves the class $\bfv$ of the fibers of $\pi\colon X\to C$; hence
$\sqrt{2}f_\bullet(\bfe_0)= \bfh+d\bfv -\sum a_i \bfe_i$ for some multiplicities $a_i\in \Z_+$.
Applied to $\sqrt{2}f_\bullet(\bfe_0)$, the invariance of the canonical 
class leads to the following constraint:
\begin{equation}
2(d-1)=\sum_i a_i.
\end{equation}
And the invariance of the intersection form gives 
\begin{equation}
2(d-1)=\sum_i a_i^2.
\end{equation}
Thus, $a_i=1$ or $0$, and there are exactly $2(d-1)$ non-zero terms in the sum $\sum_i a_i \bfe_i$. We get 
\begin{equation}
\sqrt{2}f_\bullet(\bfe_0)= \bfh+d\bfv -\sum_{i=1}^{2(d-1)} \bfe_i
\end{equation}
When $f$ is a Jonqui\`eres twist, then $\deg(f^n)\simeq n$, and the number $\bp(f^n)$ of terms in the sum also grows linearly, 
like $2\deg(f^n)$.  Thus, $\alpha(f)>0$, extending the result of Blanc and D\'eserti to all surfaces. This concludes the proof of Theorem~\ref{thm:Jonquieres-Undistorted}.

\section{Appendix: two examples}\label{par:2Examples}

\subsection{Baumslag-Solitar groups}

Fix a pair of integers $k, \ell \geq 2$. In the Baumslag-Solitar group $B_k=\langle t,x\, \vert \; txt^{-1}=x^k\rangle$, we have $\delta_x(n)\simeq\exp(n)$
(see~\cite{Gromov:AIIG}, \S~3.K1).
In the "double" Baumslag-Solitar group
\[
B_{k,\ell}=\langle t,x,y \, \vert \; txt^{-1}=x^k,xyx^{-1}=y^\ell\rangle,
\] 
we have $t^nxt^{-n}=x^{k^n}\in S^{2n+1}$ and $x^{k^n}yx^{-k^n}=y^{\ell^{k^n}}\in S^{4n+3}$; hence,
 $\delta_{y,S}(4n+3)\ge\ell^{k^n}$ and the distortion of $y$ in $B_{k,\ell}$ is at least doubly exponential. 
 In fact, we can check $\delta_y(n)\simeq\exp\exp(n)$ in $B_{k,\ell}$
as follows. Consider the homeomorphisms of the real lines $\R$ which are defined by $Y(s)=s+1$, $X(s)=\ell s$, 
and $T(s)={\mathrm{sign}}(s)\vert s\vert^k$; the relations satisfied by $t$, $x$ and $y$ in $B_{k,\ell}$ are 
also satisfied by $T$, $X$ and $Y$ in ${\mathsf{Homeo}}(\R)$: this gives a homomorphism from $B_{k,\ell}$
to ${\mathsf{Homeo}}(\R)$. If $f$ is any of the three homeomorphisms $T$, $X$ and $Y$ 
or their inverses, it satisfies $\vert f(s)\vert \leq \max(2\ell, \vert  s\vert^k)$. Thus, a recursion shows that every word $w$
of length $n$ in the generators is a homeomorphism satisfying $\vert w(0)\vert \leq (2\ell)^{k^n}$. Since $Y^m(0)=m$,
this shows that the distortion of $y$ is at most doubly exponential. 

\subsection{Locally nilpotent groups} Consider the group $M$ of upper triangular (infinite) matrices whose entries are indexed by the 
ordered set $\Q$ of rational numbers, the coefficients are rational numbers, and the diagonal coefficients are all equal 
to $1$:
\begin{enumerate}
\item $M$ is perfect (it coincides with it derived subgroup), and torsion free; 
\item $M$ is locally nilpotent (every finitely generated subgroup is nilpotent);
\item for every integer $d\geq 1$, the elementary matrix $U=\Id+ E_{0,1}$  is 
in the $d$-th derived subgroup of a finitely generated, nilpotent subgroup $N_d$ of $M$. 
\end{enumerate}
The first two assertions are described in \cite{Robinson} \S~6.2;  the last one follows from the following
two simple remarks: the elementary matrix $\Id+ E_{d,d+1}$ is in the center of the group of upper triangular 
matrices of $\SL_{d+1}(\Q)$; the translation $\alpha\mapsto \alpha-d$ is an order preserving permutation of  $\Q$, and 
this action determines an automorphism of the group $M$ that maps $\Id+ E_{d,d+1}$ to $U$.
Property~(3) implies that the distortion of $U$ in $N_d$ is $n^d$. This implies that the distortion of
$U$ in $M$ is at least $n^d$ for all $d$; but its distortion is polynomial in every finitely generated subgroup of $M$.

\bibliographystyle{plain}
\bibliography{references-disto}

\end{document}